\documentclass[a4paper,11pt]{article}
\usepackage[english]{babel}
\usepackage[latin1]{inputenc}
\usepackage{amsfonts}
\usepackage{amsmath,amsthm}
\usepackage{graphicx}
\usepackage{amssymb}
\usepackage[scale=0.8]{geometry}
\usepackage{tikz}
\usepackage{tikz-cd}
\usepackage{color}
\usepackage{float}
\usepackage{enumerate}
\usepackage{algorithmic}
\usepackage{pifont}
\usepackage[colorlinks,linkcolor=blue,urlcolor=blue,citecolor=blue]{hyperref}
\usepackage{cleveref}
\usepackage{bbold}

\usepackage{dsfont}
\usepackage{ulem}

\newtheorem{theorem}{Theorem}[section]
\newtheorem{algorithm}[theorem]{Algorithm}

\newtheorem{example}[theorem]{Example}

\newtheorem{proposition}[theorem]{Proposition}

\newtheorem{remark}[theorem]{Remark}

\newtheorem{hyp}[theorem]{Assumption}

\newcommand{\done}[1]{#1}

\newcommand{\RR}{\mathbb{R}}

\newcommand{\Prob}{\mathbb{P}}
\newcommand{\Exp}{\mathbb{E}}

\newcommand{\sL}{\mathrm{L}}

\newcommand{\calA}{{\cal A}}
\newcommand{\calF}{{\cal F}}

\newcommand{\calC}{{\cal C}}
\newcommand{\calM}{{\cal M}}
\newcommand{\calU}{{\cal U}}

\begin{document}

\title{A probabilistic reduced basis method for parameter-dependent problems}
\author{Marie Billaud-Friess\footnote{Centrale Nantes, Nantes Universit\'e, LMJL UMR CNRS 6629, Nantes, France}, Arthur Macherey$^{*,\dagger}$, Anthony Nouy$^*$, Cl\'ementine Prieur\footnote{Univ. Grenoble Alpes, CNRS, Inria, Grenoble INP*, LJK, 38000 Grenoble, France
* Institute of Engineering Univ. Grenoble Alpes}}

\date{\today}
\maketitle

\begin{abstract}
Probabilistic variants of Model Order Reduction (MOR) methods have recently emerged for improving stability and computational performance of classical approaches. In this paper, we propose a probabilistic Reduced Basis Method (RBM) for the approximation of a family of parameter-dependent functions. It relies on a probabilistic greedy algorithm with an error indicator that can be written as an expectation of some parameter-dependent random variable. Practical algorithms relying on Monte Carlo estimates of this error indicator are discussed. In particular, when using Probably Approximately Correct  (PAC) bandit algorithm, the resulting procedure is proven to be a weak greedy algorithm with high probability. Intended applications concern the approximation of a parameter-dependent family of functions for which we only have access to (noisy) pointwise evaluations. As a particular application, we consider the approximation of solution manifolds of  
 linear parameter-dependent partial differential equations with a  probabilistic interpretation through the Feynman-Kac formula.
\end{abstract}





\section{Introduction}\label{sec1}

This article focuses on the approximation of a family of functions $\calM = \{u(\xi) : \xi \in \Xi\}$  indexed by a parameter $\xi $, each function $u(\xi)$ being an element of some high-dimensional vector space $V$.
The functions $u(\xi)$ can be known a priori, or implicitly given through parameter-dependent equations. In multi-query contexts such as optimization, control or uncertainty quantification, one is interested in computing $u(\xi)$  for many instances of the parameter. For complex numerical models, this can be computationally intractable. Model order reduction (MOR) methods aim at providing an approximation $u_n(\xi)$ of $u(\xi)$ which can be evaluated efficiently for any $\xi$ in the parameter set $\Xi$. 
For linear approximation methods, an approximation $u_n(\xi)$ is obtained  by means of a projection onto a low-dimensional subspace $V_n$ which is chosen to approximate at best $ \mathcal{M}$, uniformly over $\Xi$  
for  empirical interpolation method (EIM) or reduced basis method (RBM), or in mean-square sense for proper orthogonal decomposition (POD) or proper generalized decomposition (PGD)  methods (see, e.g., the survey \cite{Nouy2017Jun}) .\\

Probabilistic variants of MOR methods have been recently proposed  for  improving stability and  computational performance of classical MOR methods. In \cite{Cohen2020}, the authors introduced a probabilistic greedy algorithm for the construction of reduced spaces $V_n$, which uses different training sets in $\Xi$ with moderate cardinality, randomly chosen at each iteration, that allows a sparse exploration of a possibly high-dimensional parameter set. 
In \cite{Cai2022Mar}, the authors derive a similar probabilistic EIM using sequential sampling in $\Xi$, which provides an interpolation with a prescribed precision with high probability. Let us also mention that a control variate method using a reduced basis paradigm has been proposed in \cite{Boyaval2010} for Monte Carlo (MC) estimation of the expectation of a collection of random variables $u(\xi)$ in a space $V$ of second-order random variables. A greedy algorithm is introduced to select a subspace $V_n$ of random variables, that relies on a statistical estimation of the projection error. This algorithm has been analyzed in \cite{Blel2021} and proven to be a weak greedy algorithm with high probability. Probabilistic approaches have also been introduced for providing efficient and numerically stable   error estimates for reduced order models \cite{Homescu2007Jun,janon2018goal,Smetana2020Dec,Smetana2019Mar}.  In  \cite{Balabanov2019Dec,Balabanov2021Mar,Balabanov2021preconditioners,Saibaba2020May,Zahm2016Apr}, 
random sketching methods have been systematically used in different tasks of projection-based model order reduction, including the construction of reduced spaces or libraries of reduced spaces, the projection onto these spaces, the error estimation and preconditioning. \\

Here, we consider the problem of computing an approximation $u_n$ of $u$ within a reduced basis framework. The reduced basis method performs in two steps, {\it offline} and {\it online}.  During the offline stage, a reduced space $V_n$ is generated  from snapshots $u(\xi_i)$ at parameter values $\xi_i$ greedily selected by maximizing over $\Xi$ (or some subset of $\Xi$) an error indicator $\Delta(u_{n-1}(\xi),\xi)$ which provides a measure of the discrepancy between $ u(\xi) $ and $u_{n-1}(\xi) $. 
Then, during the online step, $u_n(\xi)$ is obtained by some projection onto $V_n$.  
\\
In this paper, we propose a probabilistic  greedy algorithm for which  $\Delta(u_n(\xi),\xi)$  is the square error norm $\Vert u (\xi) - u_n(\xi) \Vert_V^2$, 
expressed as the expectation of some parameter-dependent random variable $Z_n(\xi)$,
\begin{equation}
\Delta(u_n(\xi),\xi) = \mathbb{E}(Z_n(\xi)).
\label{eq:errorV}
\end{equation}
\done{For maximizing $\mathbb{E}(Z_n(\xi))$ we rely on MC estimates. We consider either a naive MC approach with fixed number of samples or  a PAC (Probably Approximately Correct) bandit algorithm proposed  by the authors in \cite{Billaud22Jun} based on adaptive sampling. } The algorithm only requires a limited number of samples  by preferably sampling random variables associated with a probable maximizer $\xi$. It is particularly suitable for applications 
where the random variable $Z_n(\xi)$ is costly to sample.   
Under suitable assumptions on the distribution of $Z_n(\xi)$, it provides a PAC maximizer in relative precision, meaning that with high probability the parameter $\xi$ is a quasi-optimal solution of the optimization problem. We prove in this work that the resulting greedy algorithm is a weak-greedy algorithm with high probability.\\

Intended applications concern the approximation of a parameter-dependent family of functions $u(\xi)$ defined on a bounded domain $D$ for which we have access to (possibly noisy) pointwise evaluations $u(\xi)(x) := u(x,\xi)$ for any $x \in D$. The proposed probabilistic greedy algorithm can be used to generate a sequence of spaces $V_n$ and corresponding interpolations $u_n$ of $u$ onto $V_n$. Assuming $u(\xi) \in L^2(D)$ and we have a direct access to pointwise evaluations, the square error norm $\Vert u(\xi) - u_n(\xi) \Vert_{L^2(D)}^2$ used to select the parameter $\xi$ can be estimated  from samples of $Z_n(\xi) = \vert D \vert  \vert u(Y,\xi)-u_n(Y,\xi)\vert^2$ with $Y$ a uniform random variable over $D$. It results in a probabilistic EIM in the spirit of \cite{Cai2022Mar}.  In a fully discrete setting where $\Xi$ and $D$ are finite sets, $u$ can be identified with a matrix and the proposed algorithm is a probabilistic version of adaptive cross approximation for low-rank matrix approximation \cite{bebendorf2000approximation,Tyrtyshnikov:2000tk}, with a particular column-selection strategy. 
Another context  is the solution of a linear parameter-dependent partial differential equation (PDE) defined on a bounded  domain $D$ and whose solution $u(\xi)$ admits a probabilistic representation through the Feynman-Kac formula. This allows to express a pointwise evaluation $u(x,\xi)$  as the expectation of a functional of some stochastic process. The problem being linear, the error $u(\xi)-u_n(\xi)$ also admits a Feynman-Kac representation, which again allows to express the square error norm $\Delta(u_n(\xi),\xi)  = \Vert u (\xi) - u_n(\xi) \Vert_{L^2(D)}^2$ as the expectation of some random variable $Z_n(\xi)$ and to estimate it through Monte-Carlo simulations of stochastic processes. This is a natural framework to apply the proposed probabilistic greedy algorithm, which allows a direct estimation of the targeted error norm and avoids the use of  possibly highly biased residual based error estimates. \done{This leads to error estimators with better effectivity, which improves the behavior of weak-greedy algorithms}. In practice, as the exact solution of the PDE is not available, the snapshots used for generating the reduced space $V_n$ are numerical approximations computed from pointwise evaluations of the exact solution $u(\xi)$ by some interpolation or learning procedure. This results in a fully probabilistic setting which opens the route for the solution of high-dimensional PDEs (see, e.g. \cite{Billaud-Friess2020May} where the authors rely on interpolation on sparse polynomial spaces).       
  \\

This paper is structured as follows. In Section \ref{sec:reduced basisM}  we recall basic facts concerning reduced basis method. Then in Section \ref{sec:PGA} we present and analyze our new probabilistic greedy algorithm. Based on this algorithm, we derive  in Section \ref{sec:probaPDE} a new reduced basis method for parameter-dependent PDEs with probabilistic interpretation. Numerical results illustrating the performance of the proposed approaches are presented in Section \ref{sec:numres}.

\section{Reduced basis greedy algorithms} \label{sec:reduced basisM} 

As discussed in the introduction, reduced basis method relies on two steps. We mainly focus on the offline stage during which the reduced subspace $V_n\subset V$ is constructed. In particular, we recall in this section some basic facts concerning greedy algorithms usually considered in that context. For detailed overview on that topic see, e.g., surveys \cite{Haasdonk2017,Nouy2017morbook}.\\

Throughout this paper, $V$ is some Hilbert space equipped with a norm $\|\cdot\|_V$. We seek an approximation $u_n(\xi)$ of $u(\xi)$ in a low-dimensional space $V_n$ which 
is designed  to well approximate the solution manifold  
$$\calM = \{u(\xi) : \xi \in \Xi\}.$$
A benchmark for optimal linear approximation is given by 
the Kolmogorov $n$-width
$$
d_n(\calM)_V := \inf_{ \dim V_n = n} \sup_{u \in \calM} \|u - P_{V_n} u \|_V,
$$
where the infimum is taken over all $n$-dimensional subspaces $V_n$ of $V$ and where $ P_{V_{n}}$ stands for the orthogonal projection onto $V_{n}$. 
However, an optimal space $V_n$ is in general out of reach. 
A  prominent approach  is to rely on a greedy algorithm for generating a sequence of spaces 
from suitably selected parameter values.  
Starting from $V_0=\{0\}$, the $n$-th step of this algorithm reads as follows. Given $\{\xi_1, \dots, \xi_{n-1}\} \subset \Xi$ and the corresponding subspace 
$$
V_{n-1}= \mathrm{span}  \{u(\xi_1), \dots, u(\xi_{n-1}) \},
$$
a new parameter value $\xi_{n}$ is selected as
\begin{equation}
\|u(\xi_n) - u_{n-1}(\xi_n)\|_V = \sup_{\xi \in \Xi} \|u(\xi) - u_{n-1}(\xi)\|_V,
\label{eq:optimal}
\end{equation}
where $ u_{n-1}$ stands for an approximation of $u(\xi)$ in $V_{n-1}$. 
However, this ideal algorithm is still unfeasible in practice, at least for the two following reasons:
\begin{enumerate}[1)]
\item  computing the error $\|u(\xi) - u_{n-1}(\xi)\|_V$ for all $\xi \in \Xi$ may be unfeasible in practice (e.g. when $u(\xi)$ is only given by some parameter-dependent equation), and
\item maximizing this error over $\Xi$ is a non trivial optimization problem.
\end{enumerate}

Point 1) is usually tackled by  selecting a parameter $\xi_n$ which  maximizes some surrogate error indicator $\Delta(u_{n-1}(\xi),\xi)$ that can be easily estimated. 
Assuming $u_n(\xi)$ is a quasi-optimal projection of $u(\xi)$ onto $V_n$ and assuming 
$\Delta(u_{n-1}(\xi),\xi)$ is equivalent to $\|u(\xi) - u_{n-1}(\xi)\|_V$, 
 there exists $\gamma \in (0,1]$ such that  
\begin{equation}
\|u(\xi_n) -P_{V_{n-1}}u(\xi_n)\|_V \ge \gamma \sup_{\xi \in \Xi} \|u(\xi) -P_{V_{n-1}}u(\xi)\|_V,
\label{eq:suboptimal}
\end{equation}
which yields a weak-greedy algorithm.  Quasi-optimality means that
the approximation $u_n$ of $u$ in $V_n$ satisfies
\begin{equation}
\| u(\xi) - u_n(\xi) \|_{V} \le C \| u(\xi) - P_{V_n}u(\xi) \|_{V}
\label{eq:quasioptn}
\end{equation}
for some constant $C$ independent from $V_n$ and $\xi$.
Although the generated sequence $V_n$ is not optimal, it has been proven in \cite{Binev2011Jun,Buffa2012,Devore2013} that the approximation error 
$$
\sigma_n(\calM)_V:= \sup_{u \in \calM} \|u-P_{V_n}u\|_V
$$
has the same type of decay as the benchmark $d_n(\calM)_V$ for algebraic or exponential convergence. 

\begin{remark}\label{eq:parameq} In the case of parameter-dependent linear equation arising e.g., from the discretization of some parameter-dependent linear PDE of the form $r(u(\xi),\xi)=0$ with $ u(\xi) \in V = \RR^N$, the approximation $u_n(\xi)$ is typically obtained through some (Petrov-)Galerkin projection onto $V_n$, with a complexity depending on $n \ll N$. In such a context, a weak greedy algorithm classically involves a certified residual based error estimate $\Delta(u_n(\xi),\xi)$, that is an upper bound of the true error. However, for some applications, such an error estimate can be pessimistic (when the underlying discrete operator is badly conditionned) so that the generated sequence  $V_n$ is far from being optimal. 
A possible strategy to improve such an estimate is to consider a preconditioned residual  \cite{Cohen2012Sep,Zahm2016Apr,Balabanov2021preconditioners}.
In Section \ref{sec:probaPDE}, we overcome this limitation by considering for $\Delta(u_n(\xi),\xi)$ the targeted  square error norm $\Vert u(\xi) - u_n(\xi) \Vert_{V}^2$, which is evaluated using adaptive Monte-Carlo estimations. 
\end{remark}

Point 2) is addressed by transforming the continuous optimization problem over $\Xi$ into a discrete optimization over a finite subset  $\widetilde \Xi \subset \Xi$. Choosing the training set  $\tilde \Xi$  is a delicate task. 
As pointed out in \cite[Section 2]{Cohen2020}, if $\tilde \Xi$ is an $\varepsilon$-net
of $\Xi$, then a greedy algorithm for the approximation of the discrete solution manifold $\widetilde \calM = \{u(\xi) : \xi \in \tilde \Xi\}$ generates a sequence of spaces 
that are able to achieve a precision in $O(\epsilon)$ with similar performance as the ideal greedy algorithm. 
However, \done{the cardinality of} $\tilde \Xi $ may be very large 
for a parameter set $\Xi$ in a high-dimensional space $\RR^p$ and when a low precision $\varepsilon$ is required. 
In \cite{Cohen2020}, the authors propose a greedy algorithm which uses different training sets  randomly chosen at each step. Under suitable assumptions on the approximability of the solution map $\xi \mapsto u(\xi)$ by sparse polynomial expansions,  training sets can be chosen of moderate size independent of the parametric dimension $p$. 
\\

To conclude this section, we give a practical deterministic (weak)-greedy algorithm that can be summarized as follows. 

\begin{algorithm}[Deterministic greedy algorithm]\label{ALGO1}
Let $\tilde \Xi \subset \Xi$ be a discrete training set and $V_0 = \{ 0 \}$.\\
For $n\ge 1$ proceed as follows.
\begin{enumerate}[(Step 1.)]
\item Select $$\xi_n \in \arg \max_{\xi \in \tilde \Xi} \Delta(u_{n-1}(\xi), \xi).$$
\item Compute $u(\xi_n)$ and  update $V_n = \text{span} \{ u(\xi_1), \ldots, u(\xi_{n}) \}$. 
\end{enumerate}
\end{algorithm}

Usually, Algorithm \ref{ALGO1} is stopped when $\Delta(u_n(\xi), \xi)$ is below some target precision $\varepsilon>0$ or for a given dimension $n$.


\section{A probabilistic greedy algorithm} \label{sec:PGA} 

In this section, we motivate and present a probabilistic variant of Algorithm \ref{ALGO1}. Such an algorithm relies on the concept of  Probably Approximately Correct  (PAC) maximum. It is proven to be a weak greedy algorithm with high probability.\\

As a starting point for our work,  we assume that, \done{for any value $\xi$ in $\Xi$}, the error estimator required at each step of Algorithm \ref{ALGO1} admits the following form  
\begin{equation}
\Delta(u_n(\xi), \xi) := \|u(\xi)-u_n(\xi)\|^2_V =  \Exp (Z_n(\xi)),
\label{eq:proberr}
\end{equation}
where $Z_{n}(\xi)$ is some parameter-dependent real valued random variable, defined on the probability space $(\Omega, \calF, \Prob)$. Here, $\Exp(Z_{n}(\xi))$  is a probabilistic representation of the current square error $\|u(\xi)-u_n(\xi)\|^2_V$  depending on the targeted applications as discussed in what follows. 

\begin{example}[Estimate of the norm of approximation error] 
\label{ex:estimation}
\done{Suppose} that $u(\xi)$ belongs to $V= L^2(D)$ the Lebesgue space of square integrable functions defined on a bounded set $D \subset \RR^d$. If $\Delta(u_n(\xi), \xi) = \|u(\xi)-u_n(\xi)\|_{L^2}^2$, 
then \eqref{eq:proberr} holds 
with   $Z_n(\xi)= \vert D \vert \vert u(Y,\xi)-u_n(Y,\xi) \vert^2$ where $Y \sim {\calU}(D)$ is a random variable with uniform distribution over $D$.
\end{example}

\begin{example}[Greedy algorithm for control variate  \cite{Blel2021,Boyaval2010}]
Let \done{us} suppose that we want to compute \done{an} MC estimate of the expectation of a parameter-dependent family of random variables 
$u(\xi)$ belonging to a Hilbert space of centered second-order random variables. MC estimate is known to slowly converge with respect to the number of samples of $u(\xi)$. Variance reduction techniques based on control variates  are usually used to improve MC estimates.  
In \cite{Blel2021} the authors propose a RB paradigm to compute a control variate with a greedy algorithm of the form of Algorithm \ref{ALGO1} where $\Delta(u_n(\xi), \xi) = \mathbb{E} (Z_n(\xi)) $ with $Z_n(\xi)=\vert u(\xi)-u_n(\xi) \vert^2$ in \eqref{eq:proberr}.
\end{example}

\subsection{Main algorithm} \label{sec:PGreedy}

Solving the following optimization problem 
\begin{equation}
\label{eq:opt}
\xi_n \in \arg \max_{\xi \in \tilde \Xi}\Exp (Z_{n-1}(\xi))
\end{equation}
is in general out of reach, since $\Exp(Z_{n-1}(\xi))$ is unknown a priori or too costly to compute. Then, we propose a greedy algorithm with an approximate solution of \eqref{eq:opt}. 

\begin{algorithm}[Probabilistic greedy algorithm]\label{ALGO2}
\done{Let $\tilde \Xi \subset \Xi$ be a discrete training set.}  Starting from $V_0 = \{ 0 \}$, proceed, for $n\ge 1$, as follows. 
\begin{enumerate}[(Step 1.)]
\item Select $$\xi_n \in {\cal S}(Z_{n-1}(\xi),\tilde \Xi).$$
\item Compute $u(\xi_n)$ and  update $V_n = \text{span} \{ u(\xi_1), \ldots, u(\xi_{n}) \}$. 
\end{enumerate}
\end{algorithm}

The question is now how to choose properly the set of candidate parameter values ${\cal S}(Z_{n-1}(\xi),\tilde \Xi)$? In view of numerical applications, a first practical and naive approach is to seek $\xi_n$   maximizing the empirical mean, i.e.
$${\cal S}(Z_{n-1}(\xi),\tilde \Xi) := \arg \max_{\xi \in \tilde \Xi} \overline{Z_{n-1}(\xi)}_K$$
where $ \overline{Z_{n-1}(\xi)}_K = \frac 1K \sum_{i=1}^K (Z_{n-1}(\xi))_i$ with $K$ i.i.d. copies of $Z_{n-1}(\xi)$. 
\done{In the following, this algorithm will be called {\bf MC-greedy}.} Despite its simplicity, it is well known that such an estimate for the expectation suffers from low convergence with respect to the number of samples  leading to possible high computational costs especially if $Z_n(\xi)$ is expensive to evaluate. Moreover, nothing ensures that the returned (random) parameter $\xi_n$ is a (quasi-)optimum for \eqref{eq:opt}, almost surely or at least with high probability.\\

Instead, the so-called {\it bandit algorithms} (see, e.g., monograph \cite{Lattimore2022}) are good candidates to address  \eqref{eq:opt}. 
Here, we particularly focus on PAC bandit algorithms that for each $n$ return a parameter value $\xi_n$ which is a probably approximately correct (PAC) maximum in relative precision for $\Exp (Z_n(\xi))$ over $\tilde \Xi$ (see \cite{Billaud22Jun}). For a given  $\varepsilon \in (0,1)$ and probability $\lambda_n \in (0,1)$, letting $\xi_n^{\star} \in \arg \max_{\xi \in \tilde \Xi} \Exp[Z_{n-1}(\xi)]$, such an algorithm returns $\xi_n$ satisfying 
\begin{align}
\Prob \left(  \Exp(Z_{n-1}(\xi_n^\star)) -  \Exp(Z_{n-1}(\xi_n))\le \varepsilon   \Exp(Z_{n-1}(\xi_n^\star)) \right) \ge 1 -  \lambda_n.
\label{eq:PACsol}
\end{align}
We use the notation ${\cal S}(Z_{n-1}(\xi),\tilde \Xi) :=   \mathrm{PAC}_{\lambda_n,\varepsilon}(Z_{n-1},\tilde \Xi)$ when $\xi_n$ satisfies \eqref{eq:PACsol}. \done{The resulting greedy algorithm is called {\bf PAC-greedy}.}
In practice, the adaptive bandit algorithm in relative precision  introduced in \cite[Section 3.2]{Billaud22Jun} is particularly interesting in the case where $Z_n(\xi)$ is costly to evaluate since it preferentially samples  the random variable $Z_n(\xi)$ for the parameter values  for which it is more likely to find a maximum.  Hence, it outperforms the mean complexity of a naive approach in terms of number of generated samples. Appendix \ref{A:bandit} gives a detailed presentation of such a PAC adaptive bandit algorithm. As stated in Proposition \ref{prop:banditresult}, such an algorithm provides a PAC maximum in relative precision, that fulfills \eqref{eq:PACsol},  in the particular case where $\{Z_n(\xi), \xi \in  \tilde \Xi\}$ are random variables satisfying some concentration inequality. The interested reader can refer to \cite{Billaud22Jun} and included references for more details. 

\subsection{ Analysis of  \done{PAC-greedy algorithm}} \label{label:PACalgo}

Now, we propose and analyze a probabilistic greedy algorithm where  the parameter $\xi_n$ is a PAC maximum in relative precision for \eqref{eq:proberr} i.e. satisfying \eqref{eq:PACsol}, at each step $n$.\\

At a step $n$ of the Algorithm \ref{ALGO2}, the reduced space $V_n=\mathrm{span}\{u(\xi_1), \dots, u(\xi_{n})\}$, as well as the approximation $u_n(\xi)$ are no longer deterministic. Indeed, they are related to the selected parameters $\xi_1, \dots, \xi_n$ depending themselves  on the errors at the previous steps through i.i.d. samples of the random variables $Z_{i}(\xi)$ for all $\xi \in\Xi$ and $i <n$ (required during PAC selection of $\xi_n$). Now, we prove that Algorithm \ref{ALGO2} is a weak greedy algorithm with high probability.

\begin{theorem}\label{eq:probaweakgreedy}
Take $(\lambda_n)_{n \ge 1} \subset (0,1)^{\mathbb{N}^+}$ such that $\sum_{n\ge 1} \lambda_n = \lambda < 1$, $\varepsilon \in (0,1)$ and $\tilde \Xi \subset \Xi$ a discrete training set. Moreover, suppose that for $n \ge 1$ the approximation $u_n$ of $u$ in $V_n$ is {\it quasi-optimal} in the sense that it satisfies \eqref{eq:suboptimal} with $\Xi$ replaced by $\tilde \Xi$.
Then, Algorithm \ref{ALGO2} is a weak-greedy algorithm of parameter $\frac{\sqrt{1 - \varepsilon}}{C}$, with probability at least $1 - \lambda,$ i.e.
\begin{equation}
\Prob  \left ( \| u(\xi_n) - P_{V_{n-1}}u(\xi_n) \|_{V} \ge \frac{\sqrt{1 - \varepsilon}}{C} \max_{\xi \in \tilde \Xi} \| u(\xi) - P_{V_{n-1}}u(\xi) \|_{V} , \forall n \ge 1 \right ) \ge 1 - \lambda.
\end{equation}
\end{theorem}

\begin{proof} 
Let first introduce some useful notation. We denote by $\mathbb{P}^{n-1}(\cdot) := \Prob(  \cdot \vert  Z_{<n} ) $  the conditional probability measure  with respect to $Z_{<n}$ that denotes the collection of random variables $Z_{i}(\xi)_k$ for all $\xi \in\Xi$ and $i <n$, where $Z_{i}(\xi)_k$ are i.i.d. copies of $Z_{i}(\xi)$. The related conditional expectation is \done{$\Exp^{n-1} \left( \cdot \right)=\Exp \left( \cdot \vert  Z_{<n}  \right)$}. Now, let $A = \cap_{n\ge 1} A_n$, each event $A_n$ being defined as
\begin{align*}
A_n := \done{\left\{ \Exp^{n-1}(Z_{n-1}(\xi_n^{\star})) -  \Exp^{n-1}(Z_{n-1}(\xi_n)) \le \varepsilon \Exp^{n-1}(Z_{n-1}(\xi_n^{\star})) \right\}},
\end{align*}
 with $\xi_n^{\star} \in \arg \max_{\xi \in \tilde \Xi} \done{ \Exp^{n-1}(Z_{n-1}(\xi))}$. Then, at each step $n$ of Algorithm \ref{ALGO2}, the parameter $\xi_n$ is a PAC maximum knowing $Z_{<n}$ i.e.
 \begin{equation}
 \label{eq:PACmaxcond}
 \Prob^{n-1} (A_n) \ge 1 - \lambda_n.
\end{equation}
Finally, as $u_{n-1}(\xi)$ is completely determined by all the steps before $n$ (i.e. depending only on $Z_{<n}$), we have
\begin{equation}
\|u(\xi)-u_{n-1}(\xi)\|^2_{V} = \Exp^{n-1}(Z_{n-1}(\xi)).
\label{eq:conderror}
\end{equation}

For all $n \ge 1$, the quasi-optimality condition \eqref{eq:quasioptn} and probabilistic representation \eqref{eq:conderror} lead to
\begin{equation}\label {eq:a}
\| u(\xi_n) - P_{V_{n-1}}u(\xi_n) \|^2_{V}  \ge \frac{1}{C^2} \| u(\xi_n) - u_{n-1}(\xi_n) \|^2_{V} = \frac{1}{C^2} \done{\Exp^{n-1}(Z_{n-1}(\xi_n))}.
\end{equation}
Moreover, if $A$ holds we have for $n \ge 1$
\begin{equation}\label{eq:b}
\done{ \Exp^{n-1}(Z_{n-1}(\xi))} \ge (1 - \varepsilon)\done{\Exp^{n-1}(Z_{n-1}(\xi_n^{\star}))} 
 \ge  (1 - \varepsilon)\max_{\xi \in \tilde \Xi} \| u(\xi) - P_{V_{n-1}}u(\xi) \|^2_{V}
 \end{equation}
by definition of $\xi_n^{\star}$. Thus,  by combining \eqref{eq:a} and \eqref{eq:b} we have that $A$ implies for all $n \ge 1$
$$
\|u(\xi_n) - P_{V_{n-1}}u(\xi_n) \|_{V} \ge \frac{\sqrt{1 - \varepsilon}}{C} \max_{\xi \in \tilde \Xi} \| u(\xi) - P_{V_{n-1}}u(\xi) \|_{V}.
$$
We now estimate $\Prob(A)$
\begin{align*}
\Prob(A) =& 1 - \Prob ( \overline{A} )  \ge 1 - \sum_{n\ge 1} \Prob(\overline{A_n})  = 1 - \sum_{n\ge 1} \done{ \Exp \left(\mathds{1} _{\overline{A_n}} \right)} \\=& 1 - \sum_{n\ge 1}  \done{\Exp \left( \underbrace{\Exp \left( \mathds{1} _{\overline{A_n}} \vert Z_{<n} \right)}_{\Prob^{n-1}(\overline{A_n})} \right)}  \ge 1 - \sum_{n\ge 1} \lambda_n,
\end{align*}
where the last inequality derives from \eqref{eq:PACmaxcond}, which concludes the proof.
\end{proof}

\begin{remark} Theorem \ref{eq:probaweakgreedy} proves that Algorithm \ref{ALGO2} is a weak greedy algorithm, with probability $1-\lambda$, for the approximation of the   discrete solution manifold $\widetilde {\cal M}$. Thus, the approximation error $\sigma_n( \widetilde {\cal M})$ has the same  decay rate as $d_n( \widetilde {\cal M})$ for algebraic or exponential convergence. In the lines of  \cite{Cohen2020}, it is possible to consider also a fully probabilistic variant of Algorithm \ref{ALGO2}, in which a training set $\Xi_n$ randomly chosen is used at each step $n$ of Algorithm \ref{ALGO2} instead of $\tilde \Xi$. For a particular class of functions that can be approximated by polynomials with a certain algebraic rate, it can be proven that, for suitable chosen size of random training set $\Xi_n$, the resulting algorithm is a weak greedy algorithm with high probability with respect to the continuous solution manifold ${\cal M}$.
\end{remark}


\section{Reduced basis method for parameter-dependent PDEs with probabilistic interpretation} \label{sec:probaPDE}

We recall that this work is motivated by the approximation, in a reduced basis framework, of  a costly function  $u(\xi) : D \to \RR$ defined on the  spatial domain $ D \subset \RR^d$ depending on the parameters $\xi$ lying in $\Xi  \subset \RR^p$. 
Here we consider the problem where $u$ is the solution of a parameter-dependent PDE with probabilistic interpretation. \\

Let $D $ be an open bounded domain in $\RR^d$. For any parameter $\xi \in \Xi$, we seek $u(\xi) : \overline D\to \RR$ the solution of  the following boundary value problem, 
\begin{equation}
\begin{split}
- \calA(\xi) u(\xi)  & = g(\xi) \quad \text{in} \quad D , \\
u(\xi) & = f(\xi)  \quad \text{on} \quad \partial D,
\end{split}
\label{eq:ellipticPDEs}
\end{equation}
where $f(\xi) : \partial D  \to \RR$, $g(\xi): \overline{D } \rightarrow \RR$ are respectively the boundary condition and source term, and $\calA(\xi)$  is a linear and elliptic partial differential  operator.\\

Since the exact solution of \eqref{eq:ellipticPDEs} is not computable in general, it is classical to consider instead $u_h(\xi)$ an approximation in some finite dimensional space $V_h \subset V$ deduced from some numerical discretization of the PDE. Classical RBM applies in that context, relying on some variational principles to provide an approximation $u_n(\xi)$ of $u_h(\xi)$ in a reduced space $V_n \subset V_h$, of small dimension, obtained through greedy algorithm (see Section \ref{sec:reduced basisM}). Here, we overcome such a priori discretization of the PDE and directly address the approximation of the solution $u(\xi)$ of  \eqref{eq:ellipticPDEs}. The key idea is to use the so-called {\it Feynman-Kac representation formula} that allows to compute pointwise estimates of $u(\xi)$ for any $x \in \bar D$. This particular framework raises the following practical questions. During the offline step, how to  choose a computable error estimator $\Delta(u_n(\xi), \xi)$ required in greedy algorithm and compute the snapshots required for generating the reduced basis and related reduced space $V_n$? During the online step, how to compute the approximation $u_n$?\\

To that goal, in this section, a probabilistic RBM using only (noisy) pointwise evaluations is presented. We first recall the Feynman-Kac formula  in  Section \ref{sec:FK}. In Section \ref{prob:reducedspaceselect}, we detail a probabilistic greedy algorithm for construction of the reduced space  $V_n$ in this setting. Finally, in Section \ref{prob:reducedapprox}, we discuss possible approaches for computing the approximation $u_n(\xi)$.

\subsection{Feynman-Kac representation formula for an elliptic PDE} \label{sec:FK}
 In what follows, $W = (W_t)_{t\ge 0}$ denotes a standard $d$-dimensional Brownian motion defined on the probability space $(\Omega,\calF,\Prob)$ endowed with its natural filtration $(\calF_t)_{t\ge 0}$. For the sake of simplicity, the dependence to parameter $\xi$ is omitted in the presentation of the Feynman-Kac formula.\\

Let us consider the boundary problem \eqref{eq:ellipticPDEs}, where the partial differential  operator $\calA$ is given as
\begin{equation}
\calA = \dfrac 1 2 \sum_{i,j=1}^d (\sigma \sigma ^T)_{ij} \dfrac{\partial^2}{\partial x_i \partial x_j} + \sum_{i=1}^d b_i \dfrac{\partial}{\partial x_i} .
 \label{differentialoperator}
\end{equation}
It is the {\it infinitesimal generator}  associated to the parameter-dependent $d$-dimensional diffusion process  $X^{x} = (X^{x}_t)_{t\ge 0}$, adapted to $(\calF_t)_{t\ge 0}$,  solution of the following stochastic differential equation (SDE)
\begin{equation}
d X^{x}_t = b(X^{x}_t) dt + \sigma(X^{x}_t) d W_t, \quad X^{x}_0 = x \in \overline{D },
\label{eq:SDE}
\end{equation}
where $b(\cdot): \RR^d  \to \RR^d$ and $\sigma(\cdot) : \RR^d \to \RR^{d\times d}$ are the drift and diffusion coefficients, respectively.\\

Before recalling Feynman-Kac formula, we  introduce additional assumptions and notation. Denoting by $\|\cdot\|$ both euclidean norm on $\RR^d$ and Frobenius norm on $\RR^{d\times d}$, we first introduce the assumption that $b$ and $\sigma$ are Lipschitz continuous.
 
\begin{hyp}\label{A1} There exists a constant $0<M<+\infty$ such that for all $x,y \in \bar D$ we have
\begin{equation}
\|b(x) - b(y)\| + \|\sigma(x)- \sigma(y)\|\le M \|x-y\|.
\label{eq:a1lipschitz}
\end{equation}
\end{hyp}
Under Assumption \ref{A1}, there exists a unique {\it strong} solution to Equation \eqref{eq:SDE} (see e.g. \cite[Chapter 5, Theorem 1.1.]{Friedman2010}).\\

Denoting $a = \sigma \sigma^T$, we introduce the following uniform ellipticity assumption.
\begin{hyp}\label{A2}
There exists $c > 0$ such that
$$
y^T a(x) y \geq c \|y\|^2, \quad \text{for all } y \in \RR^d, \; x \in \overline{D }.
$$
\end{hyp}
As problem \eqref{eq:ellipticPDEs} is defined on a bounded domain, we  define the {\it first exit time} of $D $ for the process $X^x$ as 
\begin{equation}
\tau^{x} = \inf \left\{ s > 0 ~ : ~ X^{x}_s \notin {D}  \right\}.
\label{eq:exittime}
\end{equation}
Also, we assume some regularity property on the spatial domain $\overline{D}$ and data. 
\begin{hyp}\label{A3}
The domain $D $ is an open connected bounded domain of $\RR^d$, regular in the sense that it satisfies
$$
\Prob(\tau^x = 0) = 1, \quad x \in \partial D.
$$
\end{hyp}
\begin{hyp}\label{A4}
 We assume that $f$ is continuous on $\partial D $, $g$ is H\"older-continuous on $\overline{D }$.  
\end{hyp}

The following probabilistic representation theorem \cite[Chapter 6, Theorem 2.4]{Friedman2010} holds.  
\begin{theorem}[Feynman-Kac formula\label{th:FK}] Under Assumptions \ref{A1}-\ref{A4} there exists a unique solution of \eqref{eq:ellipticPDEs} in $\calC ( \overline{D } ) \cap \calC^2( D )$, which satisfies for all $x \in \overline{D }$  
\begin{equation}
\begin{split}
u(x) & = \Exp \left(   f(X^{x}_{\tau^{x}}) + \int_0^{\tau^{x}}  g(X^{x}_t) dt \right) \done{=:}~  \Exp(F(x,X^{x}) ) ,
\end{split}
\label{eq:FK}
\end{equation}
where $X^x$ is the stopped diffusion process solution of \eqref{eq:SDE}. 
\end{theorem}

Theorem \ref{th:FK} allows to derive a probabilistic numerical method for the computation of pointwise MC estimate of $u$, see Appendix \ref{probnumu}.

\subsection{Offline step}  \label{prob:reducedspaceselect}

During the offline step, the probabilistic greedy algorithm presented in Section \ref{sec:PGA}  is considered to construct the reduced space $V_n$. The keystone of such an algorithm  is the probabilistic reinterpretation of the error estimate $\Delta (u_n(\xi), \xi)$ as in Equation \eqref{eq:proberr}. Using the Feynman-Kac  representation formula, we show in Section \ref{prob:proberror} that it is possible to rewrite the square of the approximation error as an expectation. Then, in Section \ref{sec:algoimplementation}, we discuss possible strategies for practical implementation of such an algorithm.

\subsubsection{Probabilistic error estimate} \label{prob:proberror}

Let us assume that $u_n(\xi)$ is a linear approximation of $u(\xi)$ in a given reduced   space $V_n\subset V$ (e.g., obtained using Algorithm \ref{ALGO2}). We recall that $u(\xi)  \in \calC(\overline{D}) \cap \calC^2(D)$ is the unique solution of \eqref{eq:ellipticPDEs} with the following probabilistic representation 
\begin{equation}
u(x,\xi) = \Exp \left( F(x,X^{x,\xi},\xi)  \right), \quad x \in \bar D,
\label{eq:probarepxi}
\end{equation}
with $X^{x,\xi}$ the parameter-dependent stopped diffusion process solution of \eqref{eq:SDE}. In classical RB methods, the error estimate $\Delta(u_n(\xi), \xi)$ used in Algorithm \ref{ALGO1} is usually related to some suitable norm of the equation residual. Here, we follow another path by considering the $L^2$-norm of the current approximation error $e_n(\xi) = u(\xi)-u_n(\xi)$, i.e.
$$\Delta(u_n(\xi), \xi) = \| e_n(\xi) \|^2_{L^2}.$$ In what follows, we give a possible probabilistic reinterpretation of this error. Assuming that $u_n(\xi)$ is regular enough (the regularity being inherited from the snapshots), the error $e_{n}(\xi) := u(\xi) - u_n(\xi)$ is the unique solution, for all $\xi$ in $\Xi$, of
\begin{align}
\begin{aligned}
-\calA (\xi) e_{n}(\xi)  & = g_n(\xi)  \quad \text{on} \quad D, \\
e_{n}(\xi) & = f_n(\xi) \quad \text{on} \quad \partial D,
\end{aligned}
\label{chap5eq:ellipticerrorPDE}
\end{align}
where $f_{n}(\xi) := f(\xi) - u_{n}(\xi)$ and $g_{n}(\xi) = g(\xi) + \calA (\xi) u_{n}(\xi)$.
By Feynman-Kac representation theorem,  for all $\xi$ in $\Xi$, $e_{n}(\xi)$ is the unique solution of \eqref{chap5eq:ellipticerrorPDE} in $\calC(\overline{D}) \cap \calC^2(D)$ and satisfies for all $x \in \overline{D}$  
\begin{equation}
e_n(x,\xi) = \Exp \left( f_{n}(X^{x,\xi}_{\tau^{x,\xi}},\xi) + \int_0^{\tau^{x,\xi}}  g_n(X^{x,\xi}_t,\xi) dt  \right) \done{=:}~ \Exp\left( F_{n}(x,X^{x,\xi},\xi)  \right),
\label{chap5eq:FKerrorell}
\end{equation} 
with ${X}^{x,\xi}$ the stopped diffusion process solution of \eqref{eq:SDE}. Then, we have the following probabilistic reinterpretation for $\| e_{n}(\xi) \|^2_{L^2}$. 

\begin{theorem}
Let $Y \sim U(D)$ be uniformly distributed on $D$. Let $W$ and $\tilde W$ be two independent standard d-dimensional Brownian motions defined on $(\Omega, \mathcal{F}, \mathbb{P})$ and independent of $Y$. For any $x \in  \overline{D}$, let $X^{x,\xi}$ and $\tilde{X}^{x,\xi}$ be solutions of \eqref{eq:SDE} with $W$, $\tilde W$ respectively. Then we have for any $\xi$ in $\Xi$
\begin{align}
\| e_{n}(\xi) \|^2_{L^2} =   \vert D \vert \Exp \left( Z_n(\xi) \right),
\label{eq:probarepresentationerrorell}
\end{align}
with $Z_n(\xi) = F_n(Y,X^{Y,\xi},\xi)  F_n(Y,\tilde{X}^{Y,\xi},\xi)$ 
and $ \vert D \vert $ the Lebesgue measure of $D$.
\label{thm:errrorprobarepresentationell}
\end{theorem}

\begin{proof}
We first recall
$$
\| e_{n}(\xi) \|^2_{L^2}  = \int_D e_n(x,\xi)^2 dx  =   \vert D  \vert \Exp \left( e_n(Y,\xi)^2  \right).
$$
Since, for any $x$, $X^{x,\xi}$ and $\tilde{X}^{x,\xi}$ are i.i.d. random processes, we have
$$
 \Exp \left( e_n(Y,\xi)^2  \right) = \Exp \left( \Exp \left( F_n(Y,X^{Y,\xi},\xi)) \vert Y   \right)^2  \right)$$
 $$ = \Exp  \left( \Exp \left( F_n(Y,X^{Y,\xi},\xi)  \vert Y   \right) \Exp  \left( F_n(Y,\tilde{X}^{Y,\xi},\xi)  \vert  Y \right) \right), 
$$
and
$$
\| e_{n}(\xi) \|^2_{L^2}  = \vert D  \vert  \Exp \left(\Exp  \left( F_n(Y,X^{Y,\xi},\xi) F_n(Y,\tilde{X}^{Y,\xi},\xi) \vert  Y   \right) \right).
$$
Then by the law of iterated expectation we obtain \eqref{eq:probarepresentationerrorell}.
\end{proof}

\begin{remark}
Assuming the existence of probabilistic representations for the gradient of $u(\xi)$ and $u_n(\xi)$, it would be possible to consider probabilistic interpretation of other  norms of the approximation error, such as the $H^1$-norm. Such probabilistic representations have been derived in simple cases, see e.g. \cite[Corollary IV.5.2]{Gobet2016}. 
\end{remark}

\subsubsection{A probabilistic greedy algorithm using pointwise evaluations} \label{sec:algoimplementation}

For the purpose of numerical applications, we can apply Algorithm \ref{ALGO2} together with the error estimate \eqref{eq:probarepresentationerrorell} for the construction of the reduced space $V_n$.  

\paragraph{Sample computation.} The samples of $Z_n(\xi)$ (as defined in Theorem \ref{thm:errrorprobarepresentationell}) are generated from the functional $F_n$ and independent  trajectories of the discrete diffusion process $X^{Y,\Delta t}$\done{,} with $Y \sim {\cal U}(D)$. The discrete diffusion process $X^{Y,\Delta t}$ is computed using a suitable time integration scheme (see Appendix \ref{probnumu})  for the stochastic ODE \eqref{eq:SDE}. 
\paragraph{Snapshot computation.} Within Algorithm \ref{ALGO2}, the reduced space corresponds to $V_n = \mathrm{span} \{u(\xi_1),\\ \dots , u(\xi_n)\}$. However, the snapshots $\{u(\xi_1),\dots, u(\xi_n)\}$ are generally not available  since it requires to compute the exact solution of \eqref{eq:ellipticPDEs} for parameter instances $\{\xi_1,\dots, \xi_n\}$. From Feynman-Kac formula \eqref{eq:FK}, it is possible to compute MC estimates $u_{\Delta t,M}(x,\xi)$ of  $u(\xi)$ from independent realizations of the diffusion process $X^{x,\Delta t}$ starting from $x \in \bar D$ (as detailed in Appendix \ref{probnumu}). 
\done{Then a global numerical approximation can be computed in some finite dimensional linear space of dimension $N$ (potentially much larger than $n$), e.g., by interpolation or least-square method, from these MC pointwise estimates.} To compensate possible slow convergence of MC estimates, one can consider a sequential approach which uses the approximation error at each step as control variate in order to reduce the variance of MC estimates. Such a strategy has been  initially proposed in \cite{Gobet2004Dec,Gobet2006} for interpolation, and recently extended for high dimensional problems in \cite{Billaud-Friess2020May}.

\paragraph{Projection computation.}


\done{
For given $\xi$, the approximation $u_n(\xi)$  in $V_n$  can be computed by interpolation or a least-square projection using MC estimates $u_{\Delta t,M}(x,\xi)$ given by \eqref{eq:approxFK}, with suitable choice for evaluation points in $D$ (e.g., using magic points for interpolation \cite{Maday2008Sep}, or optimal sampling for least-squares \cite{cohen2017:optimal_weighted_least_squares}). 
The resulting complexity of this projection step is only linear in $n$ (up to $\log$).  }

\subsection{Online step}   \label{prob:reducedapprox}

\done{
Given  the reduced space $V_n$ obtained during the offline stage,  the approximation $u_n(\xi)$ is computed, with a complexity depending only on $n$ (and not on $N$), following  the procedure described in the projection step of Section \ref{sec:algoimplementation}.\\
}

\section{Numerical applications}
\label{sec:numres}

The aim of this section is \done{twofold}. We first illustrate the feasibility of a greedy algorithm with probabilistic error estimate for the approximation of a parameter-dependent function from its pointwise evaluations. Then, we present some numerical experiment concerning the probabilistic RBM, discussed in Section \ref{sec:probaPDE}, for the solution of parameter-dependent PDEs with probabilistic interpretation. 

\subsection{Approximation of parameter-dependent functions}

Let us consider the problem of computing an approximation $u_n(\xi)$ of $u(\xi)$, from its pointwise evaluations  at given points in $D$. Particularly, we seek $u_n(\xi)$ as the interpolation of $u(\xi)$ in the finite dimensional space $V_n \subset V$, such that
$$u_n(x_i,\xi) = u(x_i,\xi), ~x_i \in \Gamma,$$
with $\Gamma = \{x_1,\dots,x_n\}$  an unisolvant grid of suitably chosen interpolation points in $D$.
\done{Numerical experiments with least-square projection provided similar results. Thus they are not presented in this section.}

\subsubsection{Procedures for the construction of $V_n$}

 For constructing the space $V_n = \mathrm{span} \{u(\xi_1), \dots, u(\xi_n)\}$,  we compare different greedy procedures for the selection of the snapshots $u(\xi_i), i=1, \dots, n$.  First, we use the deterministic greedy Algorithm \ref{ALGO1}, for which $\Delta( u_{n-1}(\xi),\xi)$ is a numerical estimate of the $L^2$-norm of the approximation error $\|u(\xi)-u_n(\xi)\|_{L^2(D)}$ using some integration rule. This approach is confronted to probabilistic alternatives relying on probabilistic reinterpretation of the approximation error 
$$
\|u(\xi)-u_n(\xi)\|^2_{L^2(D)} =  \Exp (Z_n(\xi)).
$$
where $Z_n(\xi) = \vert D\vert \vert u_n(X,\xi)-u(X,\xi) \vert^2$, $X \sim {\cal U}(D)$, as discussed in Example \ref{ex:estimation}. In this setting, the set ${\cal S}(Z_{n-1}(\xi),\tilde \Xi)$ within Algorithm \ref{ALGO2} is obtained using either a crude MC estimate of the expectation or  adaptive bandit  algorithms discussed in Appendix \ref{A:bandit}.  When non asymptotic concentration inequalities are used, the parameter $\xi_n$ returned by Algorithm \ref{ALGO2} is a PAC maximum under suitable assumptions on the distribution of $Z_{n}(\xi)$. In particular, for any $\xi \in \tilde \Xi$, if there exist $a_{n}(\xi),b_{n}(\xi) \in \RR$ such that  $a_{n}(\xi) \le Z_{n}(\xi) \le b_{n}(\xi)$ a.s., concentration inequalities under the form \eqref{eq:concentration} hold. Having such a knowledge a priori of the distribution of $Z_n(\xi)$ and finding the bounds $a_n(\xi), b_n(\xi)$ is not an easy task. Here, since $u(\xi)$ is known, we set the following heuristic bounds
$$
a_{n}(\xi) = \min_{x \in \tilde D}  \vert u(x,\xi)- u_{n}(x,\xi)\vert^2 \text{ and }  b_{n}(\xi) = \max_{x \in \tilde D}  \vert u(x,\xi)- u_{n}(x,\xi)\vert^2,
$$
to perform our computations with $\tilde D \subset D$ a finite subset. Moreover, by Remark \ref{rk:dm},  we  have to define the sequence $(d_m)_{m\ge 1}$ with $d_m = \frac{\lambda_n}{\#\tilde \Xi} \frac{(p-1)}{p}m^{-p}$, $p =2$ and $\lambda_n = \frac\lambda n$.   We also consider a variant relying on asymptotic concentration inequality as Central Limit Theorem (CLT), which overcomes the necessity of computing any bound for $Z_n(\xi)$ and defining the sequence $(d_m)_{m\ge 1}$. These probabilistic approaches are also compared to another naive approach, in which $\xi_n$ is chosen at random in $\tilde \Xi$ (without replacement) at each step $n$ of Algorithm \ref{ALGO2}. \\

In what follows, the deterministic approach is called {\bf D-Greedy}, whereas the probabilistic ones using MC estimate and bandit algorithms are named {\bf MC-greedy} and {\bf PAC-greedy} relying on non asymptotic ({\bf Bounded}) or asymptotic concentration inequalities ({\bf CLT}). The last one is simply referred \done{to as} {\bf Random}.

\subsubsection{Numerical setting}

We perform some numerical tests with the methods discussed in the previous section for the approximation of the two subsequent functions
$$u(x,\xi) = 10x\sin(2\pi x \xi), ~ (x,\xi) \in [0,1]\times [2,4] $$ 
and, following \cite{Blel2021}, 
$$
u(x, \xi) =  \sqrt{x+ 0.1} \mathbb{1}_{[0,\xi]}(x)+ \left(\frac{x -\xi}{2\sqrt{\xi+ 0.1}} +\sqrt{\xi+ 0.1} \right)  \mathbb{1}_{[\xi,1]}(x),~(x,\xi) \in [0,1]^2.
$$
The test case related to each function will be designated by {\bf (TC1)} and {\bf (TC2)}, respectively.\\

For the numerical experiments, the training set $\tilde \Xi$ is obtained using $\# \tilde \Xi = 300$ equally spaced points in $\Xi$, similarly for $\tilde D$ made from equally spaced points in $D$ ($10000$ for (TC1), and $1000$ for (TC2)).  Then, the $\sL^2$-norm of the approximation error is estimated by trapezium rule.   Both deterministic and probabilistic greedy algorithms are stopped for given $n=20$ for (TC1) and $n=30$ for (TC2). The interpolation grid $\Gamma$ is set to be the sequence of magic points \cite{Maday2008Sep}, with respect to the basis of $V_n$.  For the probabilistic procedure with naive MC estimate, we set $K \in \{1, 50\}$.  Finally, for bandit algorithms, the stopping criterion is $\varepsilon = 0.9$ and $\lambda = 0.1$. \done{In \cite[Section 4]{Billaud22Jun}, it was observed that $\lambda$ has little influence on the number of samples $m_n(\xi)$ used by adaptive algorithm. However, an open problem is to find a $\lambda$ that gives an optimal compromise between a high probability of returning a PCA maximum and a small sampling complexity. }

\subsubsection{Numerical results}

Let us first study the quality of the approximations provided by the different approaches. Figures \ref{T13:errormean}  and \ref{T12:errormean} represent the evolution of the estimated expectation $\mathbb{E}_\xi$ and maximum, with respect to $\xi$, of the approximation error  $\|u_n(\xi)-u(\xi)\|_{L^2(\tilde D)}$  for (TC1) and (TC2).  These estimates have been computed using $100$ independent realizations of $u_n(\xi)$ obtained from uniform draws of $\xi$ in $\Xi$.  In that case, only one realisation of the probabilistic algorithms is performed for the comparison. For both test cases, D-greedy, MC-greedy and PAC-greedy procedures behave similarly with the same error decay with respect to $n$ reaching a precision around $10^{-14}$ for (TC1) and $10^{-5}$ for (TC2). Let us mention that for (TC2), the function to approximate has a slow decay of its \done{Kolmogorov} $n$-width (see e.g. discussion in \cite[Section 4.3.2]{Blel2021}) which explains higher error for larger $n$. For the random approach, the selection of interpolation points is less optimal. For first iterates it behaves similarly as other approaches, but we observe that the approximation $u_n(\xi)$ is less accurate with $n$, from around $15$ for (TC1) and $10$ for (TC2) respectively.  However, despite no guarantee on the optimality of the returned parameter $\xi_n$, the PAC algorithm with asymptotic concentration inequality and especially the MC greedy algorithm, either with a single random evaluation of the error estimate ($K=1$), lead to very satisfactory results with an error close to the deterministic interpolation approach for both test cases.

\begin{figure}[h]
\centering
\includegraphics[scale=0.9]{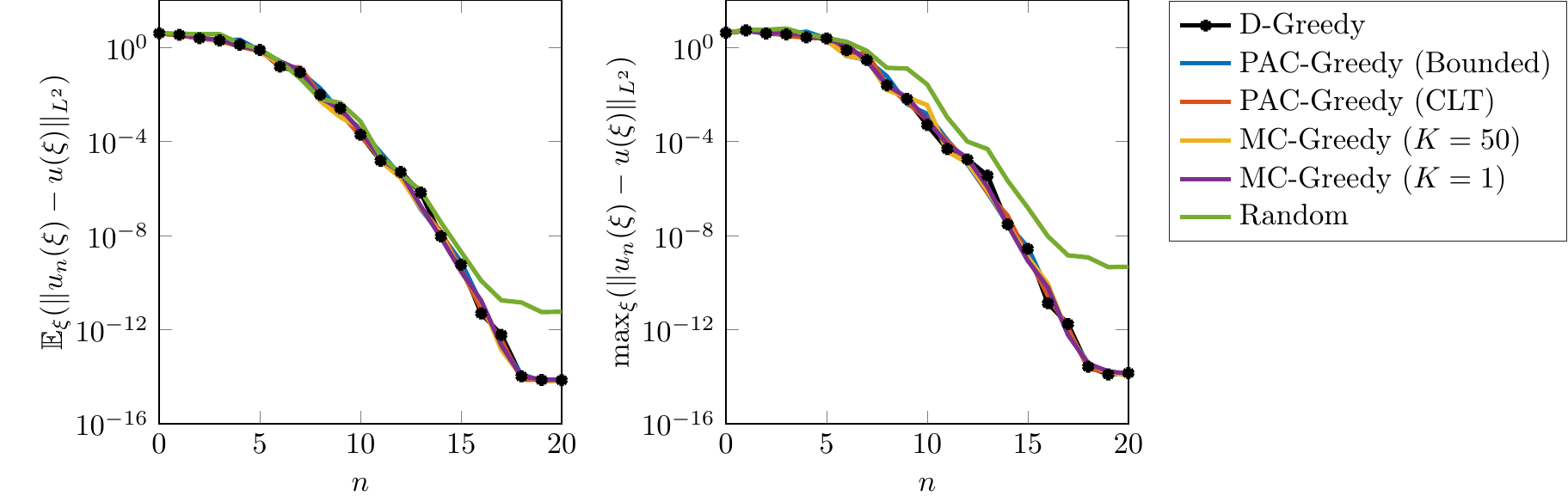}
\caption{(TC1) Evolution with respect to $n$, of the estimated expectation  and maximum of the approximation error in $L^2$-norm,  computed for $100$ instances of $\xi$, for  one realisation of the probabilistic greedy algorithms compared to the deterministic one.\label{T13:errormean}}
\end{figure}

\begin{figure}[h]
\centering
\includegraphics[scale=0.9]{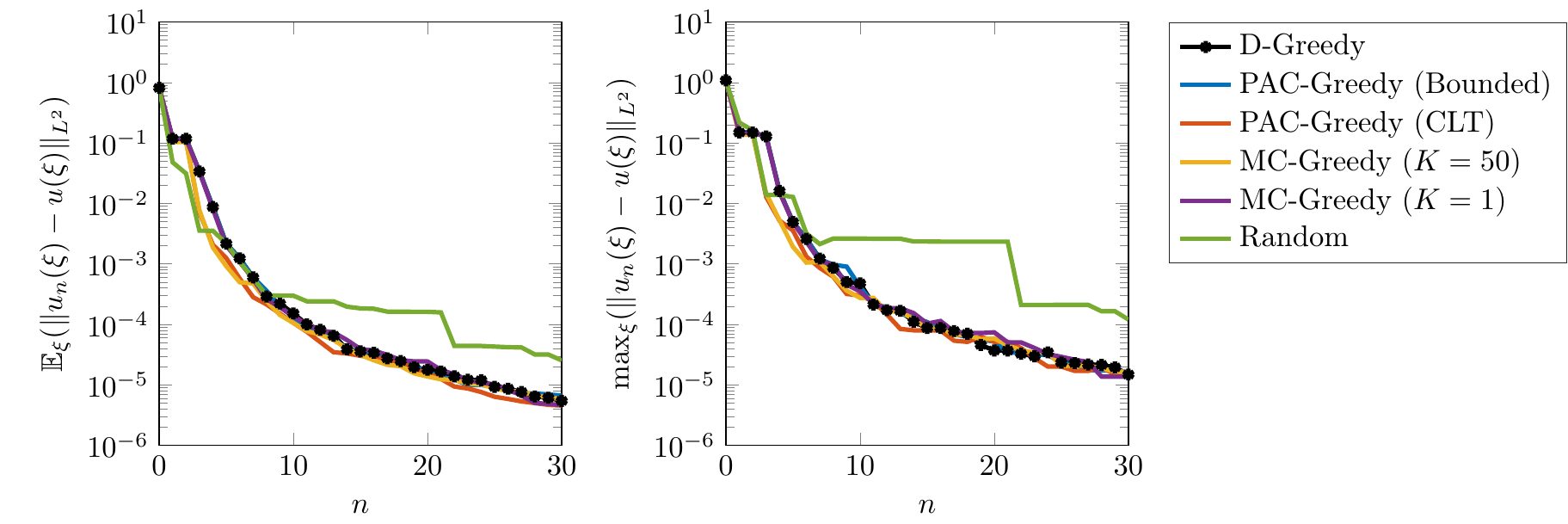}
\caption{(TC2) Evolution with respect to $n$, of the estimated expectation and maximum of the approximation error in $L^2$-norm,  computed for $100$ instances of $\xi$, for  one realisation of the probabilistic greedy algorithms compared to the deterministic one.\label{T12:errormean}}
\end{figure}

\paragraph{Greedy procedure.} We now turn to the study of the greedy procedures used for the selection of the snapshots. Figures \ref{T1:greedy}-\ref{T2:greedy} represent the error estimate $\Delta( u_{n-1}(\xi),\xi)$ as well as the number of samples $m_n(\xi)$ required during the greedy selection of $\xi_n$ for deterministic and probabilistic greedy algorithms based on bandit algorithms for (TC1) and (TC2). These curves corresponds to one realisation of the probabilistic algorithms.  \done{ First we observe that  parameters selected (indicated with the symbol $*$ on the curves) by probabilistic algorithms do not necessarily  coincide with the ones selected by D-Greedy. For MC-greedy, we observe a much higher variability of the error relatively to parameter $\xi$, for $K=1$. This is due to high variance of the estimate. However, in this simple example, even a crude MC estimate with $K=1$ allows to select a value of parameter that will make the error decrease significantly at the next iteration (see Figures \ref{T13:errormean} and  \ref{T12:errormean}).}  Second, as expected the number of samples $m_n(\xi)$ is adapted for both algorithms resulting in higher sampling in the region where it is likely to find maximum.  
Globally, we observe that PAC-greedy (CLT) works quite similarly as PAC-greedy (Bounded). But the  two approaches differ in terms of required number of samples.
Indeed, CLT based approach only requires around a maximum of $10-10^2$ samples whereas the one based on concentration inequalities requires between $10^3-10^5$ samples.

\begin{figure}[h]
\centering
\includegraphics[scale=1.15]{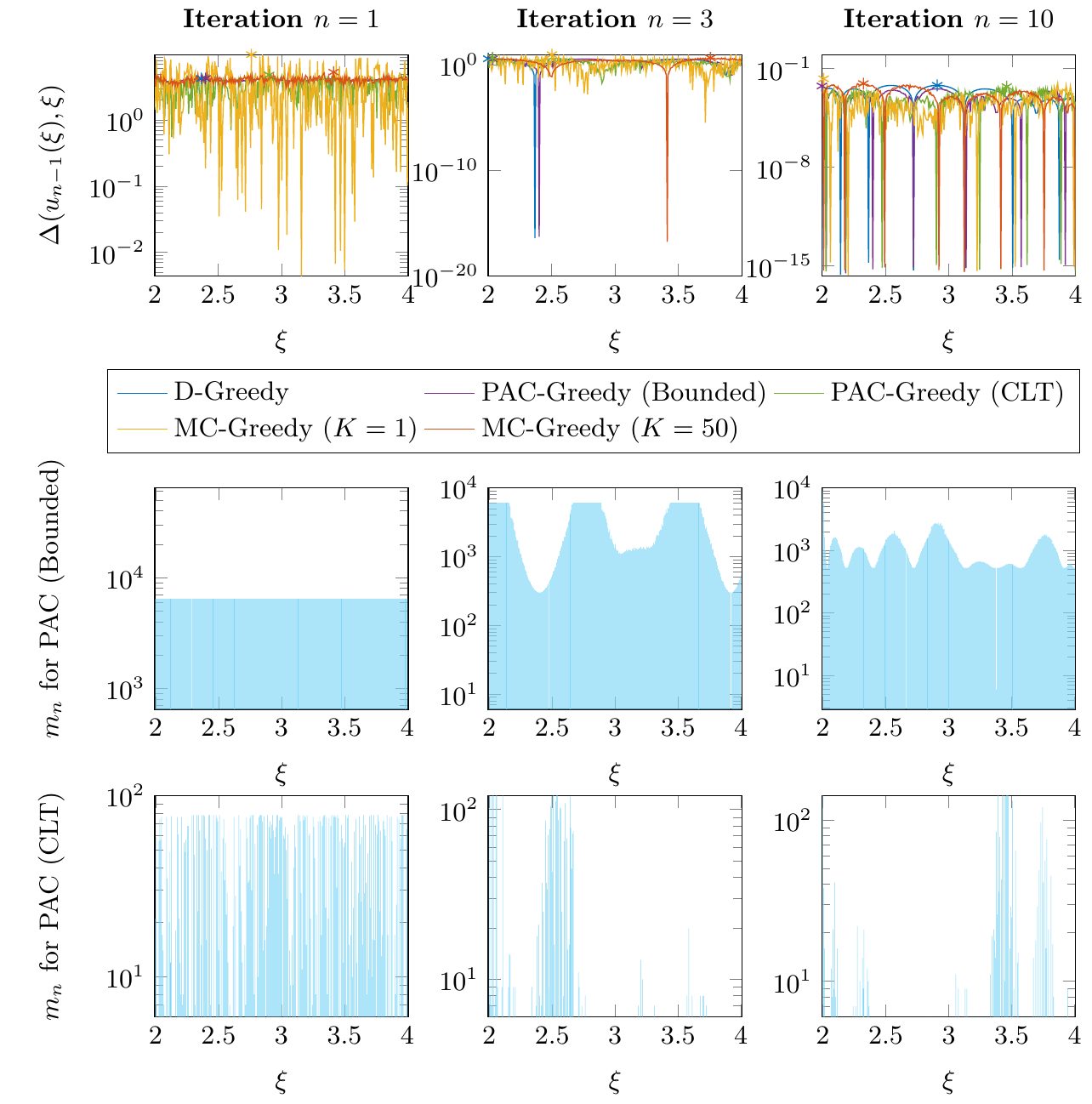}\\
\caption{(TC1) Evolution of the error during greedy procedures based on PAC bandit algorithms (top), together with required samples $m_n(\xi)$ for selecting $\xi_n \in \tilde \Xi$.\label{T1:greedy}}
\end{figure}

\begin{figure}[h]
\centering
\includegraphics[scale=1.15]{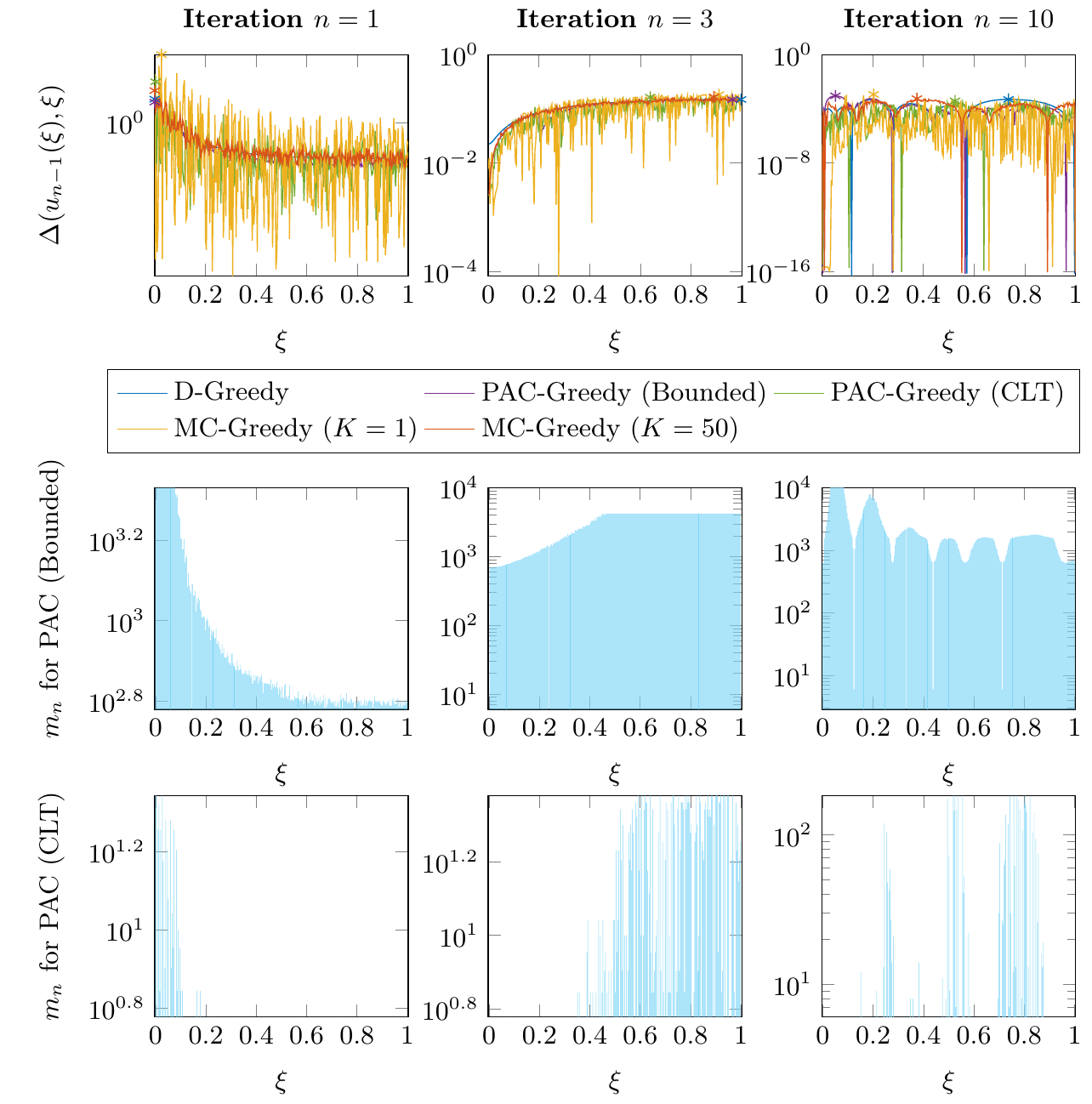}\\
\caption{(TC2) Evolution of the error during greedy procedures based on PAC bandit algorithms (top), together with required samples $m_n(\xi)$ for selecting $\xi_n \in \tilde \Xi$.\label{T2:greedy}}
\end{figure}

\paragraph{\done{Sampling complexity}.} 
\done{Now, we briefly discuss the sampling complexities of the different  methods, which are summarized   in Table \ref{tab:CPU1}. 
From Figures \ref{T13:errormean} and \ref{T12:errormean}, we see that MC- and PAC-greedy yield roughly the same level of error for a given dimension $n$. 
Therefore, the table provides the cumulated number of evaluations of the function for constructing the reduced basis of a given dimension $n$ for each test case.}

\begin{table}[h]
\centering
\done{
\begin{tabular}{l|c|c}
Method &  (TC1) for $n=20$ &(TC2) for $n=30$ \\ 
\hline \hline
D-Greedy &   $6~\times 10^7$  & $9~\times 10^6 $\\
MC-Greedy $K=1$ & $6~\times 10^3$  & $9~\times 10^3$\\
MC-Greedy $K=50$  & $3~\times 10^5$ & $4.5~\times 10^5$\\
PAC-Greedy (Bounded) &  $1.284326~\times 10^7$ &  $5.764507~\times 10^7$ \\ 
PAC-Greedy (CLT)  &   $7.507800~\times 10^4$ &$2.333380~\times 10^5$  \\
\end{tabular}
\caption{\done{Cumulative number of samples required by each algorithm, for (TC1) and (TC2), to construct reduced basis of dimension $n$. \label{tab:CPU1}}
}}
\end{table}

\done{We present as a reference the results of  D-greedy using a fine grid $\tilde D$ for numerical integration.  The most costly probabilistic algorithm is PAC-Greedy (Bounded). It requires a high number of samples for constructing non asymptotic confidence intervals, so that it does not outperform D-greedy. Let us mention that PAC-Greedy (CLT) is quite competitive with MC-Greedy with $K=50$ and underlines the interest of using some adaptive procedure for snapshot selection. However, the most interesting trade-off between efficiency and accuracy is MC-Greedy  with $K=1$.\\}

In regard to these observations, in the next section the MC-greedy approach with few samples $K$ for the error estimation will be retained in practice to reduce computational costs.

\subsection{Parameter-dependent PDE} \label{num:PDE}

Now, we focus on the solution of parameter-dependent PDEs, as introduced in Section  \ref{sec:probaPDE}.  Given $\Xi = [0.005,1]$, we seek $u(\xi), \xi \in \Xi $, solution on $D = ]0,1[$ of the following boundary problem
\begin{equation}
\begin{split}
&- {\cal A}(\xi) u(\xi) :=  -  a(\xi) u''(\xi) - b(\xi) u'(\xi) = g(\xi) \text{ in } ]0,1[, \\
& u(\xi) = f(\xi) \text{ at } x \in \{0,1\},
\end{split}
\label{eq:EDP1}
\end{equation}
for given boundary values $f(\xi) : \{0,1\} \to \RR$ and source term $g(\xi) : [0,1] \to \RR$. Moreover, we denote \done{$a(\xi) \in (0,+\infty)$} and $b(\xi) \in \RR$ the diffusion and advection coefficients respectively.  We assume that  \eqref{eq:EDP1} admits a unique solution in ${\cal C}^2([0,1])$ whose probabilistic representation is given by
$$
u(x,\xi) = \Exp(F(x,X^{x,\xi},\xi)) :=  \Exp \left(  f(X^{x,\xi}_{\tau^{x,\xi}}) + \int_0^{\tau^{x,\xi}}  g(X^{x,\xi}_t,\xi) dt \right).
$$
The associated parameter-dependent stopped diffusion process $X^{x,\xi}$ is solution of 
\begin{equation}
\label{eq:EDS1}
dX_t^{x,\xi} = b(\xi) dt + \sqrt{2a(\xi)} dW_t, \qquad X_0^{x,\xi} = x.
\end{equation}
In the following we take $a(\xi) = \xi$ and $b(\xi) = -10$.
The source term as well as Dirichlet boundary conditions are set such that the exact solution to Equation \eqref{eq:EDP1} is  
\begin{equation}
u(x,\xi) = \dfrac{\exp(x/\xi)-1}{\exp(1/\xi)-1}\cdot
\label{eq:1Dexact}
\end{equation}

\subsubsection{Compared procedures}

In what follows, we test the probabilistic RBM discussed in Section \ref{sec:probaPDE}. Since the exact solution is known, we use it for the snapshots. The projection $u_n$ is obtained from evaluations of the exact solution through interpolation ({\bf Interp}) using  magic points or a least-squares ({\bf LS}) approximation using a set of points $\tilde D$ in $D$, with $\# \tilde D \ge n$. It is also compared to the minimal residual based ({\bf MinRes}) approximation defined by
\begin{align}
\label{eq:MinRes}
u_n(\xi) = \arg \min_{v \in V_n} &\Big(\sum_{x_i \in \tilde D} \vert {\cal A}(x_i,\xi) v(x_i,\xi) +g(x_i,\xi) \vert ^2 \nonumber \\
& +  \vert  u(0,\xi) - v(0,\xi)\vert ^2 +  \vert u(1,\xi) - v(1,\xi)\vert ^2\Big).
\end{align} 
During the offline stage, we consider different greedy algorithms  for the generation of the reduced spaces. First, we perform the proposed probabilistic greedy  Algorithm \ref{ALGO2} for the construction of the reduced space $V_n$  using MC estimates with $Z_n(\xi)$ defined as in Theorem \ref{thm:errrorprobarepresentationell}. 
This approach is compared to an alternative RBM in deterministic setting. Since $u(\xi)$ is implicitly known through the boundary value problem \eqref{eq:EDP1}, the greedy selection of $V_n$ is performed using Algorithm  \ref{ALGO1} where $\Delta(u_n(\xi),\xi)$ is an estimate (using \done{trapezoidal integration rule}) of the residual based error estimate 
$\| {\cal A}(\xi) u_n(\xi) +g(\xi) \|^2_{L^2(D)}$ during the offline stage.   \\

In the presented results, the residual based  greedy algorithm is referred to as {\bf Residual}. The MC estimate using Feynman-Kac representation  is named {\bf FK-MC}. These approaches are also compared to a naive one named {\bf Random} in which the parameters $\xi_1, \dots, \xi_n$ that define  $V_n$ are chosen at random.

\subsubsection{Numerical results} \label{sec:numrespdes}

For the numerical experiments, the training set $\tilde \Xi$ is made of 200 samples from a log uniform distribution over  $\Xi$. This distribution is chosen as the solution strongly varies with respect to the viscosity $\xi$, in particular we want to  reach small viscosities. Realizations of $Z_n(\xi)$, given by \eqref{eq:approxFK}, are computed using $M=500$ realizations of the approximate stochastic diffusion process solution of Equation \eqref{eq:EDS1}, obtained by Euler-Maruyama scheme (see Appendix \ref{probnumu}) with $\Delta t = 10^{-3}$. For computing $u_n(\xi)$,  magic points are used for interpolation while for LS and MinRes approaches, we choose for $\tilde D$ a regular grid of $100$ points in $D$. Here, greedy algorithms are stopped  when $n=30$. The MC-FK greedy algorithm is performed for $K \in \{1,10\}$.\\

Figure \ref{T4:error} represents the estimated expectation and maximum, with $\xi$,  of the approximation error for the compared procedures, with respect to $n$. The  obtained results  underline the importance of both projection and  reduced basis construction.
When MinRes method is used, the approximation is less accurate than for
other deterministic approaches (dashed yellow curve). \done{Especially, when it is compared to approaches using a residual based error estimate (blue curves)  with interpolation or LS approximation. In that case, the mean approximation error is of order $10^{-11}$ for MinRes against $10^{-15}$ for the latter (for the maximum error, we have $10^{-10}$ against $10^{-14}$)}.  Second, let us comment the impact of the probabilistic reduced basis selection. As for function approximation, picking at random the reduced basis for the considered problem is far from optimal since the error expectation tends to stagnate around $10^{-10}$ for $n \ge 20$ (around $10^{-9}$ for the maximum). However, when considering residual based or FK-based error estimates (even with $K=1$),  with interpolation or least-square for projection, the error behaves quite similarly and reaches $10^{-15}$ for $n=30$. This shows that the proposed probabilistic based error procedure performs well in practice.  

\begin{figure}[h]
\centering
\includegraphics[scale=0.75]{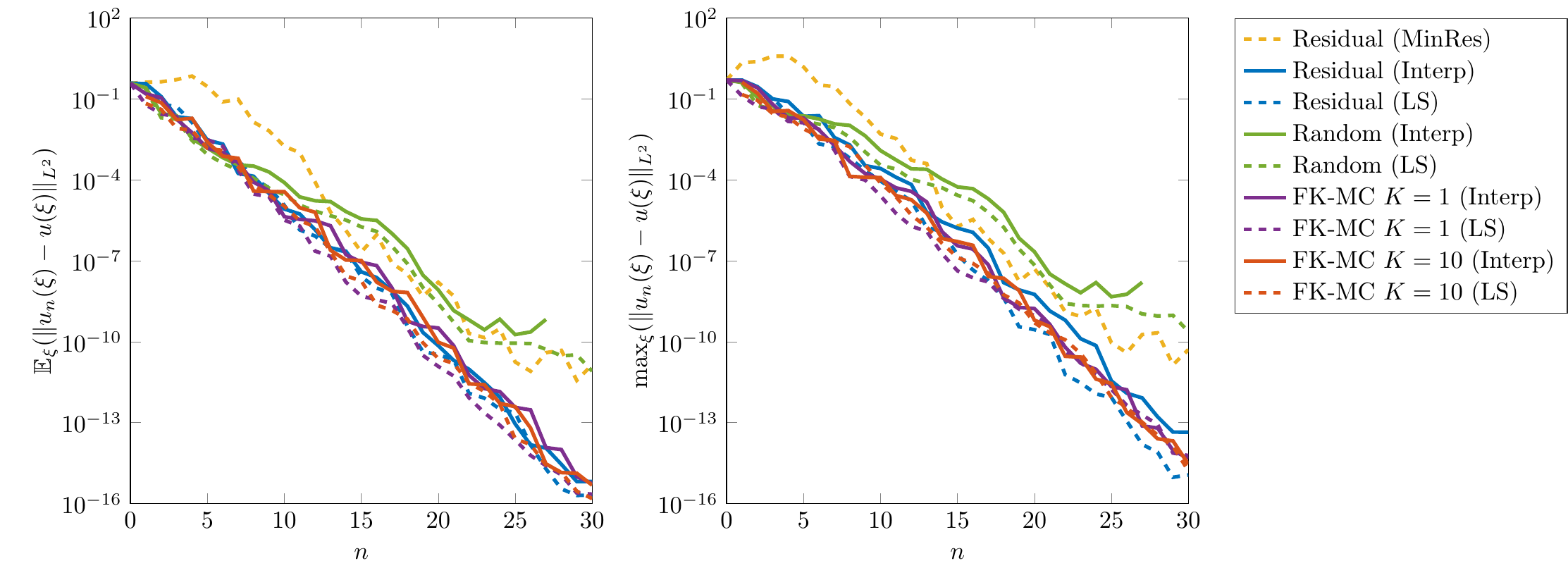}
\caption{Parameter-dependent equation: evolution with respect to $n$, of the estimated expectation and maximum of the approximation error in $L^2$-norm,  computed for $100$ instances of $\xi$, for  one realisation of the probabilistic greedy algorithms compared to the deterministic one.\label{T4:error}}
\end{figure}

\section{Conclusion} 

In this work we have proposed a probabilistic greedy algorithm for the approximation of a family of parameter-dependent functions  for which we only have access to (noisy) pointwise evaluations. It relies on an error indicator that can be written as an expectation of some parameter-dependent random variable.
Different variants of this algorithm have been proposed using either naive MC estimates or PAC bandit algorithms, the latter leading to a weak greedy algorithm with high probability. Several test cases have demonstrated the performances of the proposed procedures. \\

For parameter-dependent PDEs whose solution admits a probabilistic representation, through the Feynman-Kac formula, such an algorithm can be embedded within a probabilistic RBM using only (noisy) pointwise evaluations. Numerical results have shown the main relevance of considering Feynman-Kac error based estimate for greedy basis selection. We have also illustrated the influence of the projection during online and offline step.   \done{Indeed, we observed in Section \ref{sec:numrespdes} that interpolation or least-square projection (using pointwise evaluations of the solution) clearly outperform MinRes projection.} Following the discussion of Section \ref{sec:algoimplementation}, using a sequential procedure as proposed in \cite{Billaud-Friess2020May,Gobet2004Dec,Gobet2006} should be an interesting alternative to avoid limitations of residual based projections. However, further work should be conducted to provide a projection with  controlled error and at low cost, which is crucial for efficient model reduction.
\\

\done{The purpose of our simple numerical experiments was to illustrate the behaviour of our probabilistic greedy algorithms.  The application to more complex problems in higher dimension will be the focus of future works.}


\appendix

\section{Adaptive bandit algorithms} \label{A:bandit} 

We present an adaptive bandit algorithm to find a PAC maximum in relative precision of $\Exp(Z(\xi))$ over the discrete set $\tilde \Xi$. Here $\{Z(\xi) : \xi \in \tilde \Xi\}$ is a finite collection of random variables  satisfying $\done{\Exp(Z(\xi))} \neq 0$, defined on the probability space $(\Omega,{\cal F}, \Prob)$. After introducing some required notation, we present a practical adaptive bandit algorithm which returns a PAC maximum in relative precision for \eqref{eq:opt} when assuming suitable assumptions on the distribution of $Z(\xi)$. 

\subsection{Notations and assumptions}

 We denote by $\overline{Z(\xi)}_m$ the empirical mean of $Z(\xi)$ and $\overline{V(\xi)}_m$ its empirical variance, respectively defined by
\begin{equation}
\overline{Z(\xi)}_m = \frac{1}{m} \sum_{k=1}^{m} Z(\xi)_k \quad \text{and} \quad \overline{V(\xi)}_{m}  = \frac{1}{{m}} \sum_{k=1}^{m} \left( Z(\xi)_k - \overline{Z(\xi)}_m\right)^2,
\label{eq:empiricalestimates}
\end{equation} 
where $Z(\xi)_1 , \dots, Z(\xi)_m$ are $m$ independent  copies of $Z(\xi) $.
Moreover, the random variable $Z(\xi)$ is assumed to satisfy the following concentration inequality  
\begin{equation}
\Prob \left( \vert \overline{Z(\xi)}_m - \done{\Exp(Z(\xi))} \vert \le c(m,x,\xi) \right) \ge 1 - x,
\label{eq:concentration}
\end{equation}
for each $\xi \in \tilde \Xi$, $0 \le x \le 1$ and $m \ge 1$. The bound $c(m,x,\xi) $ depends on the probability distribution of $Z(\xi)$.\\
\begin{remark}
\label{rk:boundedness}
In view of numerical  experiments, we can consider standard concentration inequalities for sub-Gaussian or bounded random variables, see e.g. \cite[Section 2]{Billaud22Jun}. In particular, if for any $\xi \in \tilde \Xi$, there exists $a(\xi),b(\xi) \in \RR$ such that almost surely we have $a(\xi)\le Z(\xi) \le b(\xi)$, then \eqref{eq:concentration} holds with
$$
c(m,x,\xi) =\sqrt{\frac{2 \overline{V(\xi)}_{m} \log(3/x)}{{m}}} + \frac{3 \left( b(\xi) - a(\xi) \right) \log(3/x)}{{m}}.
$$
\end{remark}

An alternative to \eqref{eq:concentration} is to rely on asymptotic confidence intervals for   $\mathbb{E}(Z(\xi))$ based on limit theorems of the form 
\begin{equation}
\Prob \left( \vert  \overline{Z(\xi)}_m - \done{\Exp(Z(\xi))} \vert  \le c(m,x,\xi) \right) \to 1 - x, \text{ as $m \to \infty$.}
\label{eq:concentrationasympt}
\end{equation}
For example, for second order random variables, the central limit theorem provides such a property with   
$$
c(m,x,\xi) = \gamma_{x}\sqrt{ \frac{ \overline{V(\xi)}_{m} }{m}}, 
$$
and $\gamma_x$ the $(x/2)$-quantile of the normal distribution. 

\subsection{Algorithm} \label{app:alg}

Now, let us define a sequence $(d_m)_{m \ge 1} \subset  (0,1)^{\mathbb{N}}$, independent from $\xi$, and such that $\sum_{m \ge 1} d_m < \infty$. 
Then we introduce $c_{\xi,m} = c(m,d_m,\xi)$, and $\beta^{\pm}_{\xi,m(\xi)} = \overline{Z(\xi)}_{m(\xi)} \pm c_{\xi,{m(\xi)}}$.
Note that the concentration inequality \eqref{eq:concentration} implies that $[ \beta^-_{m(\xi)}(\xi), \beta^+_{m(\xi)}(\xi)]$ is a confidence interval for $\Exp(Z(\xi)) $ with level $1- d_{m(\xi)}$.\\

Letting $s(\xi) := \text{sign} ( \overline{Z(\xi)}_{m(\xi)} )$ and $
\varepsilon_{\xi,m(\xi)} =
 \left\{
 \begin{array}{ll}
       \frac{c_{\xi,{m(\xi)}}}{ \vert \overline{Z(\xi)}_{m(\xi)} \vert }  & \text{if } \overline{Z(\xi)}_{m(\xi)} \neq 0, \\
        +\infty& \text{otherwise.}
    \end{array}
\right.
$, we define the following estimate for $\done{\Exp(Z(\xi))}$ given by
\begin{equation}
\done{\hat{\Exp}_{{m(\xi)}}(Z(\xi))} = \left\{
    \begin{array}{ll}
        \overline{Z(\xi)}_{{m(\xi)}} - \varepsilon_{\xi,m(\xi)} ~ s
        (\xi) c_{\xi,{m(\xi)}} & \mbox{if } \varepsilon_{\xi,m(\xi)} < 1, \\
        \overline{Z(\xi)}_{m(\xi)} & \mbox{otherwise}.
    \end{array}
\right.
\label{eq:newestimate}
\end{equation}
Then, the adaptive bandit algorithm proposed in \cite{Billaud22Jun} is as follows. 
\begin{algorithm}[Adaptive bandit algorithm]
\label{bandit}~
\begin{algorithmic}[1]
\STATE Let $\varepsilon,\lambda \in (0,1)$ and $K\in \mathbb{N}$. Set $\ell=0$, $\Xi_0 = \tilde \Xi$, $m(\xi) = K$ and $\varepsilon_{\xi,m(\xi)} = +\infty$ for all $\xi \in \Xi$.
\WHILE{$\# \Xi_\ell > 1$ \textbf{and} $\underset{\xi \in \Xi_\ell}{\max} ~ \varepsilon_{\xi,m(\xi)} > \frac{\varepsilon}{2+\varepsilon}$}
\FORALL{$\xi \in \Xi_\ell$} 
\STATE Sample $Z(\xi)$, increment $m(\xi)$ and update $\varepsilon_{\xi,m(\xi)}$. 
\STATE Compute the estimate $\done{\hat{\Exp}_{{m(\xi)}}(Z(\xi))}$ using \eqref{eq:newestimate}.
\ENDFOR
\STATE Increment $\ell$ and put in $\Xi_\ell$ every $\xi \in \tilde \Xi$ such that
\begin{equation}
\beta^{+}_{\xi,m(\xi)} \ge \underset{\nu \in \tilde \Xi}{\max} ~ \beta^{-}_{\nu,m(\nu)} .
\label{eq:enrichingcondition}
\end{equation} 
\ENDWHILE
\STATE Select $\hat \xi$ such that 
$$
\hat{\xi} \in \underset{\xi \in \Xi_\ell}{\arg \max} ~ \done{\hat{\Exp}_{{m(\xi)}}(Z(\xi))}. 
$$
\end{algorithmic}
\end{algorithm} 

At each step $\ell$, the principle of Algorithm \ref{bandit} is to successively increase the number of samples $m(\xi)$ of the random variables $Z(\xi)$ in the subset  $ \Xi_\ell \subset \tilde \Xi$, obtained using confidence intervals $[ \beta^-_{m(\xi)}(\xi), \beta^+_{m(\xi)}(\xi)]$ of $\Exp(Z(\xi))$ according to \eqref{eq:enrichingcondition}. The idea behind is to use those confidence intervals to find regions of $\tilde \Xi$ where one has a high chance to  find a maximum. Then, $\hat \xi$ is returned as a maximizer over $\Xi_\ell$ of the expectation estimate defined by \eqref{eq:newestimate}. Under suitable assumptions, and when using certified non-asymptotic concentration inequalities, this algorithm returns a PAC maximum in relative precision of $\Exp(Z(\xi))$ over $\tilde \Xi$, as stated in \cite[Proposition 3.2]{Billaud22Jun}, recalled below. 

\begin{proposition} 
Let $\varepsilon,\lambda \in (0,1)$ and $(d_m)_{m \ge 1}\subset (0,1)$ be a sequence that satisfies
\begin{equation}
\sum_{m\ge 1}  d_m \le \frac{\lambda}{\# \tilde \Xi} \quad \text{and} \quad \log(1/d_m)/m \underset{m \rightarrow + \infty}{\rightarrow} 0,
\label{eq:conditionsurdmcomplete}
\end{equation}
and assume that $c$ satisfies \eqref{eq:concentration}.
If for all $\xi$ in $\Xi$, $Z(\xi)$ is a random variable  with $\done{\Exp(Z(\xi))} \neq 0$, then  Algorithm \ref{bandit} almost surely stops and returns a PAC maximum in relative precision, i.e. $\hat \xi = \mathrm{PAC}_{\lambda,\varepsilon}(Z,\tilde \Xi).$
\label{prop:banditresult}
\end{proposition}

\begin{remark} 
\label{rk:dm}
In practice a possible choice for the sequence $(d_m)_{m \ge 1}$ is to take
\begin{equation}
d_m = \delta c m^{-p} \text{ with } \delta = \frac{\lambda}{\# \tilde \Xi} \text{ and } c = \frac{p-1}{p}, \label{assumption:dmforme} 
\end{equation}
which satisfies  \eqref{eq:conditionsurdmcomplete}  for any $p>1$.
\end{remark}
In general, confidence intervals based on asymptotic theorems are much smaller  than those obtained with non-asymptotic concentration inequalities, and yield a selection of much smaller sets $\Xi_\ell$ of candidate maximizers, hence a much faster convergence of the algorithm. However, when using asymptotic theorems, we can not guarantee to obtain a PAC maximizer.

\section{Probabilistic approximation of the solution of a PDE} \label{probnumu}

%

Here we discuss the numerical computation of an estimate of $u(x)$ for any $x \in \bar D$. To that goal, we use a suitable integration scheme to get an approximation  of the diffusion process $X^x$ and  a MC method to evaluate the expectation in formula \eqref{eq:FK}. \\

An approximation of the diffusion process is obtained using a Euler-Maruyama scheme. More precisely, setting  $t_n = n\Delta t$, $n\in \mathbb{N}$, $X^{x}$ is approximated by a piecewise constant process $X^{x, \Delta t}$, where $X^{x, \Delta t}_t = X^{x, \Delta t}_n$ for $ t \in [t_n,t_{n+1}[$ and  
\begin{equation}
\begin{split}
X^{x, \Delta t}_{n+1} & = X^{x, \Delta t}_n + \Delta t ~ b(X^{x, \Delta t}_n) + \sigma(X^{x, \Delta t}_n) ~ \Delta W_n, \\
X^{x, \Delta t}_0 & = x,
\end{split}
\label{eq:Eulerscheme}
\end{equation}
where $\Delta W_n = W_{n+1} -  W_n$ is an increment of the standard Brownian motion. \\

Numerical computation of $u(x)$ for all $x \in \bar D$ requires the computation  of a stopped process $X^{x,\Delta t}$ at time  $\tau^{x,\Delta t}$, an estimation of the first exit time of $D$. Here, we consider the simplest way to define this discrete exit time
\begin{equation}
\tau^{x,\Delta t} = \min \left\{ t_n > 0 ~ : ~ X^{x,\Delta t}_{t_n} \notin {D}  \right\}.
\label{eq:discreteexittime}
\end{equation}
Such a discretization choice may lead to over-estimation of the exit time with an error in $O(\Delta t^{1/2})$. More sophisticated  approaches are possible to  improve the order of convergence, as Brownian bridge, boundary shifting or Walk On Sphere (WOS) methods, see e.g.,  \cite[Chapter 6]{Gobet2016}. These are not considered here.\\

Letting $\{ X^{x, \Delta t}(\omega_m) \}_{m=1}^M$ be $M$ independent samples of $X^{x, \Delta t}$, we obtain a  MC estimate noted $u_{\Delta t,M}(x)$ for $u(x)$ defined as
\begin{equation}
\begin{split}
 u_{\Delta t,M}(x) & = \dfrac{1}{M} \sum_{m=1}^M \bigg[ f(X^{x,\Delta t}_{\tau^{x,\Delta t}}(\omega_m)) + \int_0^{\tau^{x,\Delta t}}  g(X^{x,\Delta t}_t(\omega_m)) dt \bigg].
\end{split}
\label{eq:approxFK}
\end{equation}


\bibliography{ref}

\begin{thebibliography}{10}

\bibitem{Balabanov2019Dec}
O.~Balabanov and A.~Nouy.
\newblock {Randomized linear algebra for model reduction. Part I: Galerkin
  methods and error estimation}.
\newblock {\em Adv. Comput. Math.}, 45(5-6):2969--3019, Dec 2019.

\bibitem{Balabanov2021preconditioners}
O.~Balabanov and A.~Nouy.
\newblock Preconditioners for model order reduction by interpolation and random
  sketching of operators, 2021.

\bibitem{Balabanov2021Mar}
O.~Balabanov and A.~Nouy.
\newblock {Randomized linear algebra for model reduction. Part II: minimal
  residual methods and dictionary-based approximation}.
\newblock {\em Adv. Comput. Math.}, 47(2):26--54, Mar 2021.

\bibitem{bebendorf2000approximation}
M.~Bebendorf.
\newblock Approximation of boundary element matrices.
\newblock {\em Numerische Mathematik}, 86(4):565--589, 2000.

\bibitem{Billaud-Friess2020May}
M.~Billaud-Friess, A.~Macherey, A.~Nouy, and C.~Prieur.
\newblock Stochastic methods for solving high-dimensional partial differential
  equations.
\newblock In {\em International Conference on Monte Carlo and Quasi-Monte Carlo
  Methods in Scientific Computing}, pages 125--141. Springer, 2018.

\bibitem{Billaud22Jun}
M.~Billaud-Friess, A.~Macherey, A.~Nouy, and C.~Prieur.
\newblock A {PAC} algorithm in relative precision for bandit problem with
  costly sampling.
\newblock {\em Mathematical Methods of Operations Research}, 96(2):161--185,
  2022.

\bibitem{Binev2011Jun}
P.~Binev, A.~Cohen, W.~Dahmen, R.~DeVore, G.~Petrova, and P.~Wojtaszczyk.
\newblock {Convergence Rates for Greedy Algorithms in Reduced Basis Methods}.
\newblock {\em SIAM J. Math. Anal.}, Jun 2011.

\bibitem{Blel2021}
M.-R. Blel, V.~Ehrlacher, and T.~Leli{\ifmmode\grave{e}\else\`{e}\fi}vre.
\newblock {Influence of sampling on the convergence rates of greedy algorithms
  for parameter-dependent random variables}.
\newblock {\em arXiv:2105.14091}, May 2021.

\bibitem{Boyaval2010}
S.~Boyaval and T.~Leli{\ifmmode\grave{e}\else\`{e}\fi}vre.
\newblock {A variance reduction method for parametrized stochastic differential
  equations using the reduced basis paradigm}.
\newblock {\em Commun. Math. Sci.}, 8(3):735--762, Sep 2010.

\bibitem{Buffa2012}
A.~Buffa, Y.~Maday, A.~T. Patera, C.~Prud{'}homme, and G.~Turinici.
\newblock {A priori convergence of the greedy algorithm for the parametrized
  reduced basis method}.
\newblock {\em ESAIM: Mathematical Modelling and Numerical Analysis -
  Mod{\ifmmode\acute{e}\else\'{e}\fi}lisation
  Math{\ifmmode\acute{e}\else\'{e}\fi}matique et Analyse
  Num{\ifmmode\acute{e}\else\'{e}\fi}rique}, 46(3):595--603, 2012.

\bibitem{Cai2022Mar}
D.~Cai, C.~Yao, and Q.~Liao.
\newblock {A Stochastic Discrete Empirical Interpolation Approach for
  Parameterized Systems}.
\newblock {\em Symmetry}, 14(3):556, Mar. 2022.

\bibitem{Cohen2020}
A.~Cohen, W.~Dahmen, R.~DeVore, and J.~Nichols.
\newblock Reduced basis greedy selection using random training sets.
\newblock {\em ESAIM: Mathematical Modelling and Numerical Analysis},
  54(5):1509--1524, 2020.

\bibitem{Cohen2012Sep}
A.~Cohen, W.~Dahmen, and G.~Welper.
\newblock {Adaptivity and variational stabilization for convection-diffusion
  equations}.
\newblock {\em ESAIM: M2AN}, 46(5):1247--1273, Sep 2012.

\bibitem{cohen2017:optimal_weighted_least_squares}
A.~Cohen and G.~Migliorati.
\newblock Optimal weighted least-squares methods.
\newblock {\em SMAI Journal of Computational Mathematics}, 3:181--203, 2017.

\bibitem{Devore2013}
R.~DeVore, G.~Petrova, and P.~Wojtaszczyk.
\newblock {Greedy Algorithms for Reduced Bases in Banach Spaces}.
\newblock {\em Constr. Approx.}, 37(3):455--466, Jun 2013.

\bibitem{Friedman2010}
A.~Friedman.
\newblock {Stochastic Differential Equations and Applications}.
\newblock In {\em {Stochastic Differential Equations}}, pages 75--148.
  Springer, Berlin, Germany, 2010.

\bibitem{Gobet2016}
E.~Gobet.
\newblock {\em {Monte-Carlo Methods and Stochastic Processes: From Linear to
  Non-linear}}.
\newblock CRC Press, Boca Raton, FL, USA, 2016.

\bibitem{Gobet2004Dec}
E.~Gobet and S.~Maire.
\newblock {A spectral Monte Carlo method for the Poisson equation}.
\newblock {\em De Gruyter}, 10(3-4):275--285, Dec. 2004.

\bibitem{Gobet2006}
E.~Gobet and S.~Maire.
\newblock {Sequential Control Variates for Functionals of Markov Processes on
  JSTOR}.
\newblock {\em SIAM J. Numer. Anal.}, 43(3):1256--1275, 2006.

\bibitem{Haasdonk2017}
B.~Haasdonk.
\newblock {\em Reduced Basis Methods for Parametrized PDEs---A Tutorial
  Introduction for Stationary and Instationary Problems}, chapter~2.
\newblock SIAM, Philadelphia, PA, 2017.

\bibitem{Homescu2007Jun}
C.~Homescu, L.~R. Petzold, and R.~Serban.
\newblock {Error Estimation for Reduced-Order Models of Dynamical Systems}.
\newblock {\em SIAM Rev.}, 49(2):277--299, Jun 2007.

\bibitem{janon2018goal}
A.~Janon, M.~Nodet, C.~Prieur, and C.~Prieur.
\newblock Goal-oriented error estimation for parameter-dependent nonlinear
  problems.
\newblock {\em ESAIM: Mathematical Modelling and Numerical Analysis},
  52(2):705--728, 2018.

\bibitem{Lattimore2022}
T.~Lattimore and C.~Szepesv{\'a}ri.
\newblock {\em Bandit algorithms}.
\newblock Cambridge University Press, Cambridge, England, 2020.

\bibitem{Maday2008Sep}
Y.~Maday, N.~C. Nguyen, A.~T. Patera, and S.~H. Pau.
\newblock {A general multipurpose interpolation procedure: the magic points}.
\newblock {\em CPAA}, 8(1):383--404, Sept. 2008.

\bibitem{Nouy2017morbook}
A.~Nouy.
\newblock {\em Low-Rank Methods for High-Dimensional Approximation and Model
  Order Reduction}, chapter~4.
\newblock SIAM, Philadelphia, PA, 2017.

\bibitem{Nouy2017Jun}
A.~Nouy.
\newblock {Low-Rank Tensor Methods for Model Order Reduction}.
\newblock {\em SpringerLink}, pages 857--882, Jun 2017.

\bibitem{Saibaba2020May}
A.~K. Saibaba.
\newblock {Randomized Discrete Empirical Interpolation Method for Nonlinear
  Model Reduction}.
\newblock {\em SIAM J. Sci. Comput.}, May 2020.

\bibitem{Smetana2020Dec}
K.~Smetana and O.~Zahm.
\newblock {Randomized residual-based error estimators for the proper
  generalized decomposition approximation of parametrized problems}.
\newblock {\em Int. J. Numer. Methods Eng.}, 121(23):5153--5177, Dec 2020.

\bibitem{Smetana2019Mar}
K.~Smetana, O.~Zahm, and A.~T. Patera.
\newblock Randomized residual-based error estimators for parametrized
  equations.
\newblock {\em SIAM journal on scientific computing}, 41(2):A900--A926, 2019.

\bibitem{Tyrtyshnikov:2000tk}
E.~Tyrtyshnikov.
\newblock Incomplete cross approximation in the mosaic-skeleton method.
\newblock {\em Computing}, 64(4):367--380, 2000.

\bibitem{Zahm2016Apr}
O.~Zahm and A.~Nouy.
\newblock Interpolation of inverse operators for preconditioning
  parameter-dependent equations.
\newblock {\em SIAM Journal on Scientific Computing}, 38(2):A1044--A1074, 2016.

\end{thebibliography}

\end{document}